\documentclass[12pt,leqno]{article}
\usepackage{amssymb,amsmath,amsthm}
\usepackage[all]{xy}
\usepackage[alphabetic,lite]{amsrefs}
\numberwithin{equation}{section}

\usepackage[top=2cm,bottom=2cm,left=2cm,right=4cm,marginparsep=0.3cm,marginparwidth=3cm,includefoot]{geometry}
\usepackage{mathtools}
\usepackage{enumerate}
\usepackage{mathrsfs}
\usepackage{csf}

\usepackage[colorlinks=true, pdfstartview=FitV, linkcolor=blue, citecolor=blue, urlcolor=blue]{hyperref}

\begin{document}

\title
{Constructible sheaves and  functions  up to infinity}
%\date{}
\author{Pierre Schapira}
\maketitle
\begin{abstract}
We  introduce the category  of  b-analytic manifolds, a natural tool to define constructible sheaves and  functions up to infinity. We study with some details the operations on these objects and also recall the Radon transform for constructible functions. 
 \end{abstract}
 
 {\renewcommand{\thefootnote}{\mbox{}}
\footnote{Key words: constructible sheaves, constructible functions, subanalytic geometry, Radon transform}
\footnote{MSC:  55N99, 32B20, 32S60}
\footnote{This research  was supported by the  ANR-15-CE40-0007 ``MICROLOCAL''.}
\addtocounter{footnote}{-3}
}
  
\tableofcontents

\section*{Introduction}
Sheaf theory is a mathematical  tool to treat the dichotomy local/global and it is not surprising that it appears now as  essential in 
topological data analysis (TDA),  its use in this field appearing first in Justin Curry's thesis~\cite{Cu13}. 
Of course, sheaves have to be treated in their derived version. To illustrate this point,
see Example~\ref{exa:fourier} below.
 
 On the other-hand, sheaf theory is a very general theory, perhaps too general for applications.  In  TDA 
 one essentially encounters sheaves associated to subsets which are topologically ``reasonable'' and there is a perfectly suited framework for such sheaves,  namely that of constructible sheaves or sometimes,  on real vector spaces, piecewise linear (PL) sheaves (see~\cite{KS21a}). 
 The triangulated category of constructible sheaves over a commutative Noetherian ring $\cor$ on a real analytic manifold plays an increasing role in various fields of mathematics and is  well  understood. 
 
 To an abelian or a triangulated category, one naturally associates its Grothendieck group: any  function defined 
 on the objects of the this category,  additive with respect to exact sequences  or to distinguished triangles and   with values in a commutative group, factorizes uniquely through the Grothendieck group and, in some sense, this group contains all the additive informations of the category.  
 When $\cor$ is a field of characteristic $0$, the  Grothendieck group of the triangulated category of constructible sheaves is known to be isomorphic  to the group of  constructible functions as well as to that of Lagrangian cycles.

 Recall that constructible functions and Lagrangian cycles first appeared in the  complex analytic setting with Masaki Kashiwara~\cite{Ka73}
 and in the  algebraic setting with Robert  MacPherson~\cite{MP74}. 
 In the complex setting,  Lagrangian cycles  were studied  for their functorial properties by several people and in particular by Victor Ginsburg~\cite{Gi86} and Claude Sabbah~\cite{Sab85}. The real case was first treated  in~\cite{Ka85}. See also~\cite{KS90}*{Ch.~IX, Notes}  for an history of the subject.
 Lagrangian cycles are not so easy to describe, contrarily to constructible functions and we shall not study them here.

The Euler calculus of constructible functions has been introduced independently by Oleg~Viro (see~\cite{Vi88}) in the complex analytic setting and by the  author in the subanalytic setting (see~\cite{Sc89}). It has many applications, particularly to tomography {\em i.e.,} real Radon transform, see~\cite{Sc95} (see also  Lars Ernstr\"{o}m~\cite{Er94} for complex projective duality) and more generally  in TDA  where it appears  in particular in  sensing (see~\cite{CGR12} for a survey) and shape analysis~\cite{CMT18} and also in the study of  persistence modules through their rank invariants and their local Euler characteristic also known as Betti curve in the  community~\cite{Umeda17}.

A constructible function $\phi$ on a real analytic manifold $X$ is mathematically very simple: it is a $\Z$-valued function, the sets $\opb{\phi}(m)$ ($m\in\Z$) being all subanalytic and the family of such sets being locally finite. It is not difficult (with the tools of subanalytic geometry at hands) to check that the set $\CF(X)$ of constructible functions on $X$ is a commutative unital algebra and that the inverse image ({\em i.e.,} composition) of such a function by a  morphism of real analytic manifolds   $f\cl Z\to X$ is again constructible.

Things become more unusual when looking at direct images, in particular integration. Assume that $\phi$ has compact support. One may write $\phi$ as a finite sum $\sum_{i\in I} c_i\un_{K_i}$ where $c_i\in\Z$, $K_i$ is a compact subanaytic subset of $X$ and for $S\subset X$, $\un_{S}$ is the characteristic function of $S$. Then 
one defines the integral of $\phi$ by the formula
\eqn
&&\int_X\phi=\sum_{i\in I}c_i\cdot\chi(K_i)
\eneqn
where $\chi(K_i)$ denotes the Euler-Poincar{\'e} index of $K_i$. (This definition does not depend on the decomposition of $\phi$--see the comments after~\eqref{eq:intcst}.) For a morphism $f\cl X\to Y$ of real analytic manifolds, one defines the integral along $f$  of 
a function  $\phi\in\CF(X)$ whose support is proper with respect to $f$ by setting for $y\in Y$, 
\eqn
&&(\int_f\phi)(y)=\int_X\phi\cdot\un_{\opb{f}(y)},
\eneqn
and one checks that one obtains a constructible function on $Y$.
 This integral has all properties of classical integrals (linearity and  Fubini theorem--that is,  functoriality), except that it is not positive (the integral of $\un_{(0,1)}$ is $-1$) and a set reduced to one point has integral $1$. 
In fact, one easily translates all operations on constructible sheaves to operations on constructible functions. In particular duality makes sense for constructible functions and commutes with direct images.

Constructible sheaves and functions cause problems at infinity. For example, the set $\N$ is subanalytic in $\R$ 
(contrarily to the set $\{1/n;n\in\N\}$) but of course no finiteness properties may be obtained in this case. 
 Hence,  we  shall define the notion of being ``constructible up to infinity'', as mentioned in the title.
% and it will appear that this notion is  much more natural than the usual one. 
%For example,   Betti curves and surfaces of multi-parameters persistence modules are examples of constructible functions for the $\gamma$-topology.  
For that purpose we introduce the category of b-analytic manifolds. An object $\fX$ is  an open embedding
$X\subset \bX$ of smooth real analytic manifolds with $X$ subanalytic and relatively compact in $\bX$, and a  morphism 
$f\cl\fX\to\fY$ is a real analytic map $f\cl X\to Y$ such that the graph of $f$ is subanalytic in $\bX\times\bY$. 
Then a subset of $X$ is 
``subanalytic  up to infinity''--we shall also say ``b-subanalytic'', for short-- if it is subanalytic in $\bX$.  (As a non-example, $\N$ is not 
subanalytic up to infinity in $\R$ whatever the choice of $\R_\infty$.)
As we shall see,  this notion is  much more natural than the usual one and makes calculations easier. For example, the direct and inverse images for sheaves commute now  with duality (see below for a precise statement), non proper convolution becomes associative, etc.

 The notion of being subanalytic  up to infinity is  closely related to 
that of  definable sets and  of $o$-minimal structures, well known from the specialists (see in particular~\cites{VD98, DM96}) and constructible sheaves and functions in this  framework have already been defined in~\cites{Schu03, EP14}. Nevertheless, our approach for sheaves, based on the notion of  micro-support,  is of a different nature and provides a convenient setup to use microlocal sheaf theory while benefiting of the finiteness properties enjoyed in the framework of $o$-minimal structure
 (see for instance~\cite{CCG21}.

{\bf Sections~\ref{section:shv} and~\ref{section:cstfct}} are detailed reviews on  (derived) sheaves and constructible functions, posted here for the reader's convenience.

In {\bf Section~\ref{section:shvinfty}} we define a derived  sheaf constructible up to infinity on $X$ as a constructible sheaf  whose micro-support is subanalytic in the cotangent bundle $T^*\bX$. We shall also say, for short, that such a sheaf is ``b-constructible''. This is equivalent to saying  that its (proper or non proper) direct image in $\bX$ is again constructible. Note that such a property  already appeared  in~\cite{KS21}.
We briefly study the six operations on the triangulated category of  b-constructible sheaves.  Contrarily to the classical constructible case, the two inverse images $\opb{f}$ and $\epb{f}$ are exchanged by duality,  the two direct images $\roim{f}$ and $\reim{f}$ are constructible without any properness assumptions and, again, are  exchanged by duality.  As a nice application, we find that  non proper convolution on a real vector space $\BBV$ is well defined on constructible sheaves up to infinity and is associative. Such a non proper convolution appears when using
 the so-called $\gamma$-topology, associated with a closed convex proper cone $\gamma$ of $\BBV$. This topology, already introduced in~\cite{KS90}, plays an increasing  role in TDA.  For example,   Betti curves and surfaces of multi-parameters persistence modules are examples of constructible functions for the $\gamma$-topology.

In {\bf Section~\ref{section:fctinfty}},  we define  the space $\CF(\fX)$ of constructible functions up to infinity and study with some care the operations on such functions. 
Contrarily to the classical case, we have now two kind of integrals, proper and non proper, and, as for sheaves,  these operations are  defined without any properness  hypothesis.  Moreover, they are exchanged by duality. On a real vector space $\BBV$, we study with some care constructible functions for the $\gamma$-topology.
%Note that  Betti curves and surfaces of multi-parameters persistence modules are examples of constructible functions for the $\gamma$-topology.  It appears that the notion of being constructible up to infinity is much more natural than the usual notion.  

In {\bf Section~\ref{section:correspondence}}, posted here for easier accessibility, we recall (and adapt) the main results of~\cite{Sc95} in which we obtain an inversion formula for the Radon transform of constructible functions. This formula asserts that one can recover a constructible function on a real vector space $\BBV$ from the knowledge of the Euler-Poincar{\'e} index of its restriction to all affine hyperplanes. For example, if $\dim\BBV=3$, one can reconstruct a compact subanalytic subset from the knowledge of the number of connected components and holes of the restriction of the compact set to all slices (affine planes).

To conclude this introduction, let us recall that the Euler calculus of constructible functions already had many applications in TDA, especially under the impulse of Robert Ghrist and his collaborators (see in particular [CGR12]). Very recently, in a paper partly based on some results exposed here, Vadim Lebovici [Leb21] introduces the very promising idea of hybrid transform of constructible functions, a transform which combines classical Lebesgue integration and the Euler calculus. This new idea generalizes and unifies several previous results of specialists of TDA.

\medskip
{\bf Acknowledgement.} 
I warmly thank   Ezra Miller  for several fruitful comments on a previous version of this paper
as well as Fran{\c c}ois Petit for several stimulating discussions.
  
 \medskip
{\bf Convention.} 
In this paper, $\cor$ denotes a commutative unital Noetherian ring  with finite global dimension (see {\em e.g.}~\cite{KS90}*{exe.~I.~28}). From Section~\ref{section:cstfct} and until the end of the paper, $\cor$ is a field of characteristic $0$. 

\section{A short review on sheaves}\label{section:shv}

 In this section, we shall give a very brief overview of sheaf theory in its derived setting. 
We shall  assume that the reader has some basic notions on sheaves. In particular, we do not recall the definitions of presheaves and sheaves, neither the fundamental result which asserts that the forgetful functor, from sheaves to presheaves, admits  a left adjoint. We denote by $\Psh(\cor_X)$ the abelian category of presheaves on $X$ with values in $\md[\cor]$ and by $\md[\cor_X]$ the full abelian subcategory consisting of sheaves. Hence, by definition, a morphism of sheaves is a morphism of the underlying presheaves.
We  refer to~\cite{KS90} for a detailed exposition. 

\subsubsection*{Some notations}
Recall that a topological  space is {\em  good} if it is Hausdorff, locally compact, countable at infinity  (that is, countable union of compact subsets) and of finite flabby dimension.
 This last condition means that there exists an integer $d$ such that any sheaf admits a resolution of length  $\leq d$ by flabby sheaves. (Recall that a sheaf is flabby if any section on an open subset extends to the whole space.)  It is satisfied for example by  $C^0$-manifolds of dimension $\leq d-1$. 

For a space $X$, we  denote by $\Delta_X$  the diagonal of $X\times X$ and if $f\cl X\to Y$ is a map, we denote by $\Gamma_f$ its graph in $X\times Y$. We denote by $\rmpt$ the space consisting of a single element and by $a_X\cl X\to\rmpt$ the unique map from $X$ to $\rmpt$. 

Given topological spaces $X_i$ ($i=1,2,3$) we set
$X_{ij}=X_i\times X_j$, $X_{123}=X_1\times X_2\times X_3$.    We denote by $q_i\cl X_{ij}\to X_i$ and $q_{ij}\cl X_{123}\to X_{ij}$ the projections.
\eq\label{diag:123}
&&\ba{l}\xymatrix{
&X_{12}\ar[ld]_-{q_1}\ar[rd]^-{q_2}&   && &X_{123}\ar[ld]_{q_{12}}\ar[rd]^-{q_{23}}\ar[d]|-{q_{13}}&        \\
   X_1&&X_2                                    & & X_{12}&X_{13}&X_{23}                                                          
}
\ea\eneq
For $A\subset X_{12}$ and $B\subset X_{23}$, one sets 
\eq\label{diag:123B}
&&A\times_{2}B=A\times_{X_2}B=\opb{q_{12}}A\cap\opb{q_{23}}B,\quad A\conv[2] B=q_{13}(A\times_{2}B).
\eneq

\subsubsection*{Basic operations on sheaves}

We consider a commutative unital Noetherian  ring $\cor$  of finite global dimension. However, assuming that  $\cor$ is a field would be sufficient  for most applications.

Let us first only consider sheaves, passing to the derived categories later.

Given two sheaves $F$ and $G$ on $X$, one defines their {\bf tensor product} $F\tens G$ as the sheaf associated to the presheaf $U\mapsto F(U)\tens G(U)$, 
($U$ open in $X$).

The {\bf internal hom},  denoted $\hom$, is the presheaf $U\mapsto \Hom(F\vert_U,G\vert_U)$ where $\Hom$ is taken in the category $\Psh(\cor_U)$ and it appears that this presheaf is a sheaf as soon as $G$ is a sheaf.
One proves (\cite{KS90}*{Prop.~2.2.9}) that $(\tens,\hom)$ is a pair of adjoint functors, that is, for three sheaves $F,G,H$
\eqn
&&\Hom(F\tens G,H)\simeq \Hom(F,\hom(G,H)),
\eneqn
functorially in $F,G,H$.

Now consider a continuous map $f\cl X\to Y$. If $F$ is a sheaf on $X$, its {\bf direct image} denoted $\oim{f}F$ is the presheaf on $Y$ which, to $V$ open in $Y$, associates $F(\opb{f}V)$. It is easily checked that this presheaf is a sheaf. 

The {\bf inverse image} is more delicate. If $G$ is a sheaf on $Y$, one first defines it inverse image as a presheaf, $\popb{f}G$, as follows. 
For $U$ open in $X$, $\popb{f}G(U)=\sindlim G(V)$ where $V$ ranges through the family of open subset of $Y$ such that $U\subset \opb{f}V$.
Then the inverse image $\opb{f}G$ is the sheaf associated with the presheaf $\popb{f}G$. 
One proves (\cite{KS90}*{Prop.~2.3.3}) that $(\opb{f},\oim{f})$ is a pair of adjoint functors, that is 
\eqn
&&\Hom(\opb{f}G,F)\simeq \Hom(G,\oim{f}F),
\eneqn
functorially in $F,G$.
Hence $\opb{f}$ is right exact and $\oim{f}$ is left exact. In fact, 
$\opb{f}$ is exact.

As a combination of these functors we get the external product. In the situation of~\eqref{diag:123}, for $F_i\in \md[\cor_{X_i}]$, $i=1,2$, one sets
\eq\label{eq:externprod}
&&F_1\etens F_2\eqdot\opb{q_1}F_1\tens\opb{q_2}F_2.
\eneq

One denotes by $\cor_X$ the {\bf constant sheaf} on $X$ with stalk $\cor$. It is defined  as $\cor_X=\opb{a_X}\cor$, after having identified $\md[\cor]$ and $\md[\cor_{\rmpt}]$. 
The sheaf $\cor_X$ is also the sheaf of locally constant functions on $X$ with values in $\cor$. One defines similarly the sheaf $M_X$ for $M\in\md[\cor]$. 

Consider a  closed subset $S$ of $X$ and denote by $j_S\cl S\into X$ the embedding. One sets 
$\cor_{XS}\eqdot \oim{j_S}\cor_S$. This is  the sheaf on $X$ of functions with values in $\cor$ and which are locally constant on $S$ 
 and $0$ elsewhere. Now set $U=X\setminus S$. One defines the sheaf $\cor_{XU}$ by the exact sequence
\eqn
&&0\to \cor_{XU}\to\cor_X\to\cor_{XS}\to 0.
\eneqn
A locally closed set $Z$ is the (non unique)  intersection of a closed set $T$ and an open set $V$. One sets
$\cor_{XZ}\eqdot\cor_{XT}\tens\cor_{XV}$, this last sheaf depending uniquely on $Z$. One often writes $\cor_Z$ instead of $\cor_{XZ}$, especially when $Z$ is closed in $X$.
For a sheaf $F$ on $X$, one then sets
\eqn
&&F_Z\eqdot F\tens\cor_{XZ}.
\eneqn
Note that the functor $\scbul \tens\cor_{XZ}$ is exact. Moreover, if $U$ is open in $Z$ and $S$ is closed in $Z$, there are natural morphisms $F_U\to F_Z$ and $F_Z\to F_S$. 

Assuming that $X$ and $Y$ are good topological spaces, there is also a notion of {\bf proper direct image} denoted $\eim{f}F$.  
It is defined as follows, for $F$ a sheaf on $X$:
\eqn
&&\eim{f}F=\indlim[U]\oim{f} F_U
\eneqn
where $U$ ranges over the family of open subsets of $X$ such that the map $f$ is proper on $\ol U$. Hence, $\eim{f}F$ is a subsheaf of $\oim{f}F$.
In particular, if $f$ is proper on $X$ (or better, on $\supp(F)$), then $\eim{f}F\isoto\oim{f}F$. 

One checks (\cite{KS90}*{Prop.~2.5.4}) that if $Z$ is locally closed in $X$, denoting by $j_Z\cl Z\into X$ the embedding, then 
$F_Z\simeq\eim{j_Z}\opb{j_Z}F$. 

\subsubsection*{The six Grothendieck operations} 

Sheaf theory takes its full strength when treated in the derived setting, the preceding functors being replaced with their derived
version.  We denote by $\Derb(\cor_X)$ the bounded derived category of sheaves of $\cor$-modules on $X$ and simply calls an object of this category ``a sheaf''. An object of $\Derb(\cor_X)$ may be represented by  a bounded complex of sheaves
$F^\scbul$ and  a quasi-isomorphism $u\cl F^\scbul\to G^\scbul$ becomes  an isomorphism in $\Derb(\cor_X)$.
(A quasi-isomorphism is a morphism which  induces isomorphisms on the cohomology objects.) Note that morphisms  of  $\Derb(\cor_X)$ are not easy to describe.

The bifunctor $\tens$ being right exact, one has to replace it with its left derived functor $\ltens$, 
 and similarly with the functor $\etens$ that one replaces with its left derived functor $\letens$. . By the hypothesis that the ring $\cor$ has finite global dimension, the derived functor applied to objects of the bounded derived category  takes its values in   this category.

The functors 
$\oim{f}$, $\eim{f}$ and the bifunctor $\hom$ being left exact, one has to replace them with their right derived versions, 
$\roim{f}$, $\reim{f}$ and  $\rhom$. To calculate a right derived functor, for example $\roim{f}F$, the  recipe is to represent $F$ by 
a complex of injective sheaves and to apply $\oim{f}$ to this complex. 

Let us illustrate the strength of the derived approach with an example. 
\begin{example}\label{exa:fourier}
Consider  a real finite dimensional vector space $\BBV$ and a closed  proper cone $\gamma$ with vertex at $0$. Denote by $\gamma^\circ$ the polar cone in $\BBV^*$. This last cone  is convex and only  allows us to recover the convex hull of  $\gamma$. However, if one replaces $\gamma$ with  the sheaf $\cor_\gamma$  
 and replaces  the polar cone with  the Fourier-Sato transform  (see~\cite{KS90}*{\S~3.7}) of  $\cor_\gamma$,  a transform which uses the six Grothendieck operations, 
 then no information is lost and one recovers $\cor_\gamma$, hence  the initial cone $\gamma$, even if this cone  is not convex.  
\end{example}

Let us come back to the non derived operations described above. Taking the derived functors we get two pairs of adjoint functors
\eqn
&&(\ltens,\rhom),\quad (\opb{f},\roim{f}).
\eneqn
The functor $\eim{f}$ does not have an adjoint but the functor $\reim{f}$ has a right adjoint (see~\cite{KS90}*{\S~3.1}) 
\eqn
&&\epb{f}\cl \Derb(\cor_Y)\to\Derb(\cor_X)
\eneqn
and we get the pair of adjoint functors (in the derived categories) 
\eqn
&&(\reim{f},\epb{f}).
\eneqn
On a topological manifold $X$, the dualizing complex $\omega_X$  is defined by 
 $\omega_X\eqdot\epb{a_X}\cor_{\{\rmpt\}}$.  One  proves (see~\cite{KS90}*{\S~3.3})  that
 \eqn
 &&\omega_X\simeq\ori_X\,[\dim X]
 \eneqn
 where $\ori_X$ is the orientation sheaf on $X$, $\dim X$ is the dimension of  $X$ and $\ori_X\, [\dim X]$ is the shifted object.
 We shall  encounter  the duality functors
\eqn
&&\RD'_X(\scbul)=\rhom(\scbul,\cor_X),\quad \RD_X=\rhom(\scbul,\omega_X).
\eneqn

\subsubsection*{Kernels}\label{sect:Kconv}
For good topological spaces $X_i$'s as above, one often calls an object $K_{ij}\in\Derb(\cor_{X_{ij}})$ {\em a kernel}.
One defines  as usual the   convolution (one also says ``composition'')  of kernels 
\eq\label{eq:convker}
&&K_{12}\cconv[2] K_{23}\eqdot \reim{q_{13}}(\opb{q_{12}}K_{12}\ltens\opb{q_{23}}K_{23}).
\eneq
If there is no risk of confusion, we write $\conv$ instead of $\cconv[2]$. 

It is easily checked, and well known, that convolution is associative, namely given three kernels $K_{ij} \in\Derb(\cor_{X_{ij}})$, $i=1,2,3$, $j=i+1$ one has 
an isomorphism 
\eq\label{eq:assockernels}
&&(K_{12}\conv K_{23})\conv K_{34}\simeq K_{12}\conv (K_{23}\conv K_{34}),
\eneq
this isomorphism satisfying natural compatibility conditions that we shall not make here explicit. 

Of course, this construction applies in the particular case where $X_i=\rmpt$ for some $i$.  In this case, let us change our notations 
to $X_1=X$ and $X_2=Y$.
If  $K\in\Derb(\cor_{X\times Y})$ and $F\in\Derb(\cor_X)$, one usually sets $\Phi_K(F)=F\conv K$. Hence
\eq\label{eq:PhiK}
&&\Phi_K(F)= F\conv K=\reim{q_2}(\opb{q_1}F\ltens K). 
\eneq
We shall also use the right adjoint of the functor $\Phi_K(\cdot)$, namely the functor $\Psi_K(\cdot)$ (see~\cite{KS90}*{\S~3.6}), defined for $G\in\Derb(\cor_Y)$ by:
\eq\label{eq:PsiK}
&&\Psi_K(G)=\roim{q_1}\rhom(K,\epb{q_2}G).
\eneq

Hence:
\eqn
&&\RHom[\Derb(\cor_Y)](\Phi_K(F),G)\simeq \RHom[\Derb(\cor_X)](F,\Psi_K(G)).
\eneqn

 For $K\in\Derb(\cor_{X\times Y})$, set $K^v=\oim{v}K$ where $v$ is the map $X\times Y\isoto Y\times X$, $(x,y)\mapsto(y,x)$. 

\begin{lemma}\label{le:opbyker}
Let $f\cl X\to Y$, $F\in\Derb(\cor_X)$ and  $G\in\Derb(\cor_Y)$. Set for short $K_f=\cor_{\Gamma_f}$. Then 
\eqn
&\opb{f}G\simeq K_f\conv G=\Phi_{K_f^v}G,& \quad \roim{f}F\simeq \roim{q_2}\rhom(K_f,\epb{q_1}F)=\Psi_{K_f^v}F,\\
&\reim{f}F\simeq F\conv K_f=\Phi_{K_f}F,&\quad \epb{f}G\simeq \roim{q_1}\rhom(K_f,\epb{q_2}G)=\Psi_{K_f}G.
\eneqn
\end{lemma}
\begin{proof}
The first and third isomorphisms are obvious (identify $X$ with $\Gamma_f$). The two others follow by adjunction. 
\end{proof}
\begin{remark}

One may also define the  non-proper  convolution of kernels by the formula below, similar to~\eqref{eq:convker}
\eq\label{eq:npconvker}
&&K_{12}\npconv[2] K_{23}\eqdot \roim{q_{13}}(\opb{q_{12}}K_{12}\ltens\opb{q_{23}}K_{23}).
\eneq
However, one  should be aware that, in general, this operation is no more associative.

Let $f\cl X\to Y$ be as above and 
denote by $j\cl \Gamma_f\into X\times Y$ the embedding of the graph of $f$.  By remarking that the composition 
$q_1\circ j\cl \Gamma_f\to X$ is an isomorphism, we get:
\eqn
\roim{f}F&\simeq& \roim{q_2}\rhom(K_f,\epb{q_1}F) \simeq  \roim{q_2}\eim{j}\epb{j}\epb{q_1}F\\
&\simeq& \roim{q_2}\eim{j}\opb{j}\opb{q_1}F\simeq  \roim{q_2}(\opb{q_1}F\ltens\cor_{\Gamma_f})\simeq F\npconv K_f.
\eneqn
\end{remark}

\subsubsection*{Micro-support}
Now assume that $X$ is a real manifold of class $C^\infty$ and denote by $\pi_X\cl T^*X\to X$ its cotangent bundle. To $F\in\Derb(\cor_X)$, one associates its {\em micro-support} $\SSi(F)$  (also called {\em singular support}), a closed $\R^+$-conic subset of $T^*X$ and this set is co-isotropic (in a sense that we do not recall here). See~\cite{KS90}*{Th.~6.5.4}.

\subsubsection*{Subanalytic subsets}
From now on and unless otherwise specified, we work on  real analytic manifolds. 
However, almost all results extend to the case of subanalytic spaces 
for the definition of which we refer to~\cite{KS16}*{\S~2.4}.

We shall not review here the history of subanalytic geometry, which takes its origin in the work of 
 Lojasiewicz, simply mentioning the names of Gabrielov and Hironaka. References are made to~\cite{BM88}.

Let $X$ be a real analytic manifold. Denote by $\shs_X$ the family  of subanalytic subsets of $X$.
Then  $\shs_X$  is a Boolean algebra which contains
the  family of semi-analytic subsets (those locally defined by analytic inequalities) and is  closed
under  taking the closure and the interior. If $f\cl X\to Y$ is subanalytic,  $A\in\shs_X$, $B\in\shs_Y$, then 
$\opb{f}(B)\in\shs_X$ and if $f$ is proper on the closure of $A$,  then $f(A)\in\shs_Y$. 

Moreover, to be subanalytic in $X$ is a local property on $X$. More precisely, given $X=\bigcup_{a\in A}U_a$ an open covering, a subset $Z\subset X$ is subanalytic in $X$ if and only  if $Z\cap U_a$ is subanalytic in $U_a$ for all $a\in A$.

Note that if  $Z$ is a locally closed subanalytic subset of $X$, then there exist an open set $U$ and a closed subset $S$ both subanalytic in $X$ such that $Z=U\cap S$.
Indeed, set $Y=\ol Z\setminus Z$. Then $Y$ is closed since $Z$ is locally closed. Choose $S=\ol Z$  and $U=X\setminus Y$. 

A subanalytic stratification of $X$ 
is a locally finite   partition  
$X=\bigsqcup_{a\in A}X_a$ where each $X_a$ is a smooth locally closed real analytic submanifold of $X$ subanalytic in $X$, 
 and  for all $a,b\in A$, $X_a\cap \ol{X_b}\neq\varnothing$ implies $X_a\subset \ol{X_b}$.

\subsubsection*{Constructible sheaves}
A sheaf $F\in\Derb(\cor_X)$ is weakly $\R$-constructible if there exists a subanalytic stratification  $X=\bigsqcup_{a\in A}X_a$ such that for all $j\in\Z$, $H^j(F)\vert_{X_a}$ is locally constant. If moreover, these locally constant sheaves are  finitely generated  (recall that $\cor$ is Noetherian), then $F$ is $\R$-constructible. By the results of~\cite{KS90}*{Ch.~VIII}, $F$ is weakly $\R$-constructible if and only if $\SSi(F)$ is contained in a closed conic subanalytic isotropic subvariety  of $T^*X$ and this implies that  $\SSi(F)$ is equal to a 
closed conic subanalytic Lagrangian subvariety.

One denotes by $\Derb_\Rc(\cor_X)$ the full triangulated subcategory of $\Derb(\cor_X)$ consisting of $\R$-constructible sheaves.
The categories of constructible sheaves  are  closed under  the six Grothendieck operations with the exception of direct images which should be proper on the supports of the constructible sheaves.

\section{Constructible  sheaves up to infinity}\label{section:shvinfty}

\subsection{Subanalytic subsets up to infinity}\label{subsectFP}

 In order to define subanalytic subsets up to infinity, we introduce the category of b-analytic manifolds,  inspired by  (but rather different from)  that of bordered space of~\cite{DK16}. As mentioned in the introduction,  the notion of being subanalytic up to infinity is a particular case of that of definable set,  well known from the specialists 
(see~\cites{VD98, DM96}), and constructible sheaves in this framework have already been defined  in~\cites{Schu03, EP14}. 
However, our approach is direct and quite different since it is based on the notion of micro-support.

\begin{definition}\label{def:bspace}
The category of \emph{b-analytic manifolds} is the category defined as follows.
\banum
\item
An object $\fX$  is a  pair $(X,\bX)$ with $X\subset \bX$ an open embedding of real analytic manifolds such that $X$ is relatively compact and subanalytic in $\bX$.
\item
A morphism $f\cl \fX=(X,\bX) \to \fY=(Y,\bY)$ of b-analytic manifolds is a morphism of real analytic manifolds  $f\colon X\to Y$ such that
the graph $\Gamma_f$ of $f$ in $X\times Y$    is subanalytic in $\bX\times\bY$.
\item
The composition $(X,\bX) \to[f] (Y,\bY)\to[g](Z,\bZ)$ is given by $g\circ f\cl X\to Z$ and 
the identity $\id_{(X,\bX)}$ is given by $\id_{X}$  (see Lemma~\ref{lem:bordcom} below).
\eanum
If there is no risk of confusion, we shall often denote by $j_X\cl X\into\bX$ the open embedding.
\end{definition}

\begin{remark}
Instead of requiring $\bX$ to be a smooth real analytic manifold and $X$ relatively compact in it, one could ask $\bX$ to be a compact 
subanalytic space in the sense of~\cite{KS16}*{\S~2.4}. However, Definition~\ref{def:shvconstinfty} below should be modified
 by using uniquely properties (c) and (d) of Lemma~\ref{le:KS21}. 
 One could also define the notion of a b-subanalytic space. 
\end{remark}
Remark that in general, contrarily to the case of bordered spaces,  neither $(X,X)$ nor $(\bX,\bX)$ are b-analytic manifolds. However, if $X$ is compact, $(X,X)$ is a b-analytic manifold.

\begin{lemma}\label{lem:bordcom}
\banum
\item
The identity $\id_{(X,\bX)}$ is a morphism of b-analytic manifolds.
\item
Let $f\colon (X,\bX) \to (Y,\bY)$ and $g\colon (Y,\bY) \to (Z,\bZ)$ be morphisms of b-analytic manifolds. Then the composition $g\circ f$ is a morphism of b-analytic manifolds.
\eanum
\end{lemma}
\begin{proof}
(a) Since $X$ is subanalytic in $\bX$, $X\times X$ is  subanalytic in $\bX\times\bX$, and 
$\Delta_X=X\times X\cap\Delta_{\bX}$  is  subanalytic in $\bX\times\bX$. 

\spa
(b) By the hypothesis, $\Gamma_g$ is subanalytic and relatively compact in $\bY\times\bZ$ and $\Gamma_f$ is subanalytic and relatively compact in $\bX\times\bY$. It follows that 
 $\Gamma_f\times_{\bY}\Gamma_g$ is subanalytic  and relatively compact in $\bX\times\bY\times\bZ$. Therefore, its projection 
 $\Gamma_f\conv\Gamma_g$ is subanalytic in  $\bX\times\bZ$. Since $\Gamma_f\conv\Gamma_g=\Gamma_{g\circ f}$, the proof is complete.
 (Note that one could also have applied Proposition~\ref{pro:opersainfty} below.)
\end{proof}

\begin{definition}\label{def:sainfty}
Let $\fX=(X,\bX)$ be a b-analytic manifold and let 
 $Z$ be a  subset of $X$. We say that $Z$ is subanalytic  up to infinity
  if $Z$ is subanalytic in $\bX$. 
  We shall also say for short that  $Z$ is b-subanalytic.
\end{definition} 
Note the following remarks.
\begin{itemize}
\item
The property of being subanalytic up to infinity depends on the choice of $\fX$ and such a choice is supposed to have been made when using this terminology.
\item
Given $X$, there does not always exist $\fX$. As an example (of non-existence), choose $X=\N$, a real analytic manifold of dimension $0$. 
\item
The family of subsets  subanalytic  up to infinity inherits  all of the  properties of the family of subanalytic subsets with the exception that this property is no more local (but it is local for finite coverings). In particular, this family is  closed under  interior, closure, complement, finite unions and finite intersections
 and $X$ itself is subanalytic up to infinity (once $\fX$ exists).
 \end{itemize}

On a real analytic manifold $X$, the subanalytic topology and the site $\Xsa$ are defined in~\cite{KS01}. 
\begin{definition}
Let $\fX=(X,\bX)$ be a b-analytic manifold. 
\banum
\item
We shall denote by $\Op_{\Xsai}$ the category of open subsets of $X$ subanalytic up to infinity, the morphisms being the inclusions.
\item
We endow  $\Op_{\Xsai}$ with a Grothendieck topology as follows. A family $\{U_i\}_{i\in I}$ of objects of  $\Op_{\Xsai}$ 
is a covering of $U\in \Op_{\Xsai}$ if $U_i\subset U$ for all $i\in I$ and there exists $J\subset I$ with $J$ finite such that $U=\bigcup_{j\in J}U_j$.
\item
We denote by $\Xsai$ the site so obtained.
 \eanum
\end{definition}
Note that the category  $\Op_{\Xsai}$ is  closed under product of two elements (namely, the intersection of two open subsets) and admits a terminal object, namely  $X$. This makes the study of sheaves on $\Xsai$ particularly easy.

In the sequel, for $U\in\Op_{\Xsai}$, we shall denote by $U_\infty$ the b-analytic manifold $(U,\bX)$ where the embedding $j_U\cl U\into\bX$ is the composition of $j_X$ and the embedding $U\into X$.

\begin{proposition}\label{pro:opersainfty}
Let $\fXi=(X_i,\bXi)$ \lp$i=1,2,3$\rp\, be three b-analytic manifolds. 
\banum
\item
Setting $\bX_{12}=\bX_1\times\bX_2$, the pair $(X_{12},\bX_{12})$ is a b-analytic manifold. Moreover, if 
$S_1$ and $S_2$ are two  b-subanalytic  subsets of $X_1$ and $X_2$ respectively, then $S_1\times S_2$ is b-subanalytic  in $X_{12}$.
\item
Let $S_{1}$ and $S_2$ be two b-subanalytic subsets of $X_{12}$ and $X_{23}$ respectively, then $S_1\conv[2]S_2$ is 
b-subanalytic in $X_{13}$.
\item
In particular, let $f\cl \fX\to \fY$ be a morphism of b-analytic manifolds. If $Z\subset Y$ is  b-subanalytic, then $\opb{f}(Z)$ is b-subanalytic in $X$ 
and if $S\subset X$ is  b-subanalytic, then ${f}(S)$ is b-subanalytic    in $Y$.
\eanum
\end{proposition}
We shall denote by $(X\times Y)_\infty$ the b-analytic manifold  $(X\times Y,\bX\times\bY)$.
\begin{proof}
(a) is obvious.

\spa
(b) $S_1\times_{X_2}S_2$ is subanalytic and relatively compact in $\bX_{123}$. Therefore, its image by $q_{13}$ is 
 subanalytic and relatively compact in $\bX_{13}$. 
 
\spa
(c) By the hypothesis, $\Gamma_f$ is subanalytic  up to infinity in $\bX\times\bY$.  By (b), 
$\opb{f}(Z)=\Gamma_f\conv[Y]Z$ is subanalytic up to infinity in $X$ and $f(S)=S\conv[X]\Gamma_f$ is subanalytic up to infinity in $Y$.
 \end{proof}

\subsection{Constructible sheaves up to infinity}

 Constructible sheaves up to infinity can be regarded as a generalization of the notion of tame multiparameter persistence modules. In this section, we consider  b-analytic manifolds $\fX=(X,\bX)$ and $\fY=(Y,\bY)$.

\subsubsection*{Definitions}
Let  $F\in\Derb_\Rc(\cor_X)$. 
Recall that  the micro-support $\SSi(F)$ of $F$ is  a closed $\R^+$-conic  subanalytic Lagrangian subset of $T^*X$.

\begin{lemma}[{See~\cite{KS21}*{Th.2.2}}]\label{le:KS21}
Let  $F\in\Derb_\Rc(\cor_X)$. The following conditions are equivalent.
\banum
\item
The micro-support $\SSi(F)$ is subanalytic in $T^*\bX$.
\item
The micro-support $\SSi(F)$ is contained in a locally closed $\R^+$-conic  subanalytic isotropic subset  of $T^*\bX$. 
\item
$\eim{j_X}F\in\Derb_\Rc(\cor_{\bX})$.
\item
$\roim{j_X}F\in\Derb_\Rc(\cor_{\bX})$.
\eanum
\end{lemma}
\begin{proof}
For the reader's convenience, we recall the proof of loc.\ cit,   a proof which  uses the notion of a $\mu$-stratification 
(see~\cite{KS90}*{Def.~8.3.19}). 
Note that in loc.\ cit. the statement was formulated slightly differently.

\spa
(a)$\Rightarrow$(b) is obvious.  

\spa
(c)$\Rightarrow$(a) and (d)$\Rightarrow$(a)  follow from the fact that $T^*X$  is subanalytic in $T^*\bX$. Indeed, set either $\Lambda=\SSi(\eim{j_X}F)$
or $\Lambda=\SSi(\roim{j_X}F)$. 
Then  $\Lambda$ is subanalytic in $T^*\bX$ and  $\SSi(F)=\Lambda\cap T^*X$ is still subanalytic in $T^*\bX$.

\spa
(b)$\Rightarrow$(c). Assume that  $\SSi(F)$ is contained in a locally closed $\R^+$-conic  subanalytic isotropic subset  $\Lambda$ of $T^*\bX$.
By~\cite{KS90}*{Cor.~8.3.22}, there exists a $\mu$-stratification $\bX=\bigsqcup_{a\in A}Y_a$ such that 
$\Lambda\subset\bigsqcup_{a\in A}T^*_{Y_a}\bX$. 

Set $X_a=X\cap Y_a$. Then $X=\bigsqcup_{a\in A}X_a$ is a  $\mu$-stratification  and one can 
apply loc.\ cit.\ Prop.~8.4.1. Hence, for each $a\in A$, $F\vert_{X_a}$ is locally constant of finite rank. 
Hence  $(\eim{j_X}F)\vert_{X_a}$ as well as $(\eim{j_X}F)_{\bX\setminus X}\simeq0$ is
locally constant of finite rank. 
Hence $\eim{j_X}F\in \Derb_\Rc(\cor_{\bX})$. 

\spa
(c)$\Rightarrow$(d).  Using the implication (b)$\Rightarrow$(c), we get that  $\reim{j_X}\cor_X$ belongs to $\Derb_\Rc(\cor_{\bX})$.
Set $G=\eim{j_X}F$. Then $\roim{j_X}F\simeq \rhom(\reim(j_X)\cor_X,G)$ belongs to 
$\Derb_\Rc(\cor_{\bX})$  by~\cite{KS90}*{Prop.~8.4.10}.
\end{proof}

\begin{definition}\label{def:shvconstinfty}
Let $F\in\Derb_\Rc(\cor_X)$. 
One says that $F$ is constructible up to infinity if it satisfies one of the equivalent conditions in Lemma~\ref{le:KS21}. We denote by 
$\Derb_\Rc(\cor_{\fX})$ the full triangulated subcategory of $\Derb_\Rc(\cor_X)$  consisting of sheaves constructible up to infinity.

We shall also say, for short, that $F$ is ``b-constructible'' instead of ``constructible up to infinity''.
\end{definition}
It follows that if $F\in\Derb_\Rc(\cor_{\bX})$, then $\opb{j_X}F\in \Derb_\Rc(\cor_{\fX})$.

\begin{example}
Piecewise linear sheaves (\PL sheaves) on a real vector space $\BBV$ are defined in~\cite{KS21a}*{Def.~2.3}.
Clearly, $\PL$-sheaves are constructible up to infinity. 
\end{example}

\subsubsection*{Operations}

\begin{proposition}\label{pro:opersainftyint}
Let $\fX$ and $\fY$ be two b-analytic manifolds. 
\bnum
\item
Let $F\in\Derb_\Rc(\cor_{\fX})$ and $G\in\Derb_\Rc(\cor_{\fY})$. Then $F\letens G\in\Derb_\Rc(\cor_{\fXY})$.
\item
Let $F_1$ and $F_2$ belong to $\Derb_\Rc(\cor_{\fX})$. Then $F_1\ltens F_2$ and $\rhom(F_1,F_2)$ belong to $\Derb_\Rc(\cor_{\fX})$.
In particular, the dual $\RD_XF$ of  $F\in\Derb_\Rc(\cor_{\fX})$ belongs to  $\Derb_\Rc(\cor_{\fX})$.
\enum
\end{proposition}
\begin{proof}
 
All the statements follow from the similar ones for usual constructible sheaves and the isomorphisms:
\eqn
\eim{j_{X\times Y}}(F\letens G)&\simeq&\eim{j_X}F\letens\eim{j_Y}G,\\
\eim{j_X}(F_1\ltens F_2)&\simeq& \eim{j_X}F_1\ltens\eim{j_X}F_2,\\
\eim{j_X}\rhom(F_1,F_2)&\simeq& \rhom(\eim{j_X}F_1,\eim{j_X}F_2).
\eneqn
The proof of the first isomorphism is left as an exercise. The second one follows from the projection formula:
\eqn
\eim{j_X}F_1\ltens\eim{j_X}F_2&\simeq & \eim{j_X}(F_1\ltens\opb{j_X}\eim{j_X}F_2)\simeq \eim{j_X}(F_1\ltens F_2).
\eneqn
The third isomorphism follows by applying $\eim{j_X}$ to the isomorphism 
\eqn
\rhom(F_1,F_2)&\simeq& \epb{j_X}\rhom(\eim{j_X}F_1, \eim{j_X}F_2),
\eneqn
using the fact that, $j_X$ being an open immersion, $\epb{j_X}\circ\eim{j_X}\simeq\id$. 
\end{proof}

\begin{proposition}\label{pro:opersainftyext}
Let $f\cl \fX\to \fY$ be a morphism of b-analytic manifolds. 
\bnum
\item 
Let $G\in\Derb_\Rc(\cor_{\fY})$. Then 
$\opb{f}(G)$ and $\epb{f}G$ belong to $\Derb_\Rc(\cor_{\fX})$.
\item
Let  $F\in\Derb_\Rc(\cor_{\fX})$. Then $\reim{f}F$ and $\roim{f}F$ belong to $\Derb_\Rc(\cor_{\fY})$.
\enum
\end{proposition}

\begin{proof}
Let $K_f=\cor_{\Gamma_f}$. Then  $K_f\in\Derb_\Rc(\cor_{\fXY})$.
By Proposition~\ref{pro:opersainftyint} and Lemma~\ref{le:opbyker}, we are reduced to prove that

\spa
(a) if $H\in\Derb_\Rc(\cor_{\fXY})$, then 
$\reim{q_1}H$ and $\roim{q_1}H$ belong to $\Derb_\Rc(\cor_{\fX})$,

\spa
(b) if $F\in \Derb_\Rc(\cor_{\fX})$, then $\opb{q_1}F$ and $\epb{q_1}F$ belong to $\Derb_\Rc(\cor_{\fXY})$.

\spa
The assertion (b) follows from Proposition~\ref{pro:opersainftyint} since  $\opb{q_1}F\simeq F\etens\cor_Y$ and  $\epb{q_1}F\simeq F\etens\omega_Y$. To prove (a), denote by ${\widehat {q_1}}$ the projection $\bXY\to \bX$. 
Then $\reim{q_1}H\simeq \opb{j_X}\reim{{\widehat {q_1}}}\reim{j_{X\times Y}}H$ and similarly 
$\roim{q_1}H\simeq \opb{j_X}\roim{{\widehat {q_1}}}\roim{j_{X\times Y}}H$.
\end{proof}
\begin{remark}
We see in Proposition~\ref{pro:opersainftyext} an important difference between constructible sheaves and constructible sheaves up to infinity.
Indeed, for usual constructible sheaves, the (proper or non proper) direct image is no more constructible in general.
\end{remark}

\begin{corollary}\label{cor:opersainftyext}
Let $f\cl \fX\to \fY$ be a morphism of b-analytic manifolds and let $F\in\Derb_\Rc(\cor_{\fX})$ and $G\in\Derb_\Rc(\cor_{\fY})$. 
Then $\roim{f}F\simeq\RD_Y\reim{f}\RD_XF$ and $\epb{f}G\simeq \RD_X\opb{f}\RD_YG$.
\end{corollary}
\begin{proof}
(i) Both $\RD_XF$ and $\reim{f}\RD_XF$ are $\R$-constructible. Then apply~\cite{KS90}*{Exe.~VIII.3}.

\spa
(ii) Similarly, both $\RD_YG$ and $\opb{f}\RD_YG$  are $\R$-constructible. Then apply loc.\ cit.
\end{proof}
Consider b-analytic manifolds  $\fXi=(X_i,\bXi)$, ($i=1,2,3$), and kernels $K_{ij}\in\Derb_\Rc(\cor_{\fXij})$, $i=1,2$, $j=i+1$.
We have already defined in~\eqref{eq:convker} the   convolution  of kernels $K_{12}\cconv[2] K_{23}$. 

Applying  Propositions~\ref{pro:opersainftyint} and~\ref{pro:opersainftyext}, we get:
\begin{corollary}\label{cor:assocnpstar} 
In the preceding situation, $K_{12}\conv[2] K_{23}$ belongs to $\Derb_\Rc(\cor_{X_{13\infty}})$. 
\end{corollary}
Recall that the  convolution of kernels is associative (see~\eqref{eq:assockernels}). 

\subsubsection*{Base change formula and projection formula}
Consider two morphisms $f\cl \fX\to \fZ$ and $g\cl \fY\to \fZ$ of b-analytic manifolds and 
consider\footnote{In~\cite{Sc20}*{v1, v2} it was made reference to the notion of a Cartesian square in the category of b-analytic manifolds, a notion which should have been defined more precisely and that we avoid here.} 
 a Cartesian square of topological spaces
\eq\label{diag:cart1}
&&\ba{l}\xymatrix{
W\ar[r]^-{f'}\ar[d]_-{g'}&Y\ar[d]_-{g}\\
X\ar[r]_f&Z.
}\ea\eneq
Recall that the square is Cartesian means that $W$ is  isomorphic  to the space
$\{(x,y)\in X\times Y;f(x)= g(y)\}$.
We consider $W$ as a closed subanalytic subset of $X\times Y$.

The classical  base change formula for sheaves (see for example~\cite{KS90}*{Prop.~2.6.7})  together with Proposition~\ref{pro:opersainftyext} gives:

\begin{proposition}\label{pro:basechinftyshv}
Consider the  Cartesian  square~\eqref{diag:cart1} and let $F\in\Derb(\cor_{\fX})$. Then
\eq\label{eq:basechform}
&&\opb{g}\reim{f}F \simeq \reim{f'}\opb{g'}F  \mbox{ in }\Derb(\cor_{\fY}).
\eneq
\end{proposition}
Remark that the left hand-side of this isomorphism belongs to $ \Derb(\cor_{\fY})$  by the preceding results and this implies 
that the same is true for the right hand-side.

Similarly, the classical projection formula   together with Proposition~\ref{pro:opersainftyext}  gives:

\begin{proposition}\label{pro:projforminftyshv}
Let $f\cl \fX\to \fY$ be a morphism of b-analytic manifolds, let  $F\in\Derb(\cor_{\fX})$ and $G\in\Derb(\cor_{\fY})$. Then
\eq\label{eq:projforminfty}
&&\reim{f}(F\ltens\opb{f}G)\simeq \reim{f} F\ltens\ G \mbox{ in }\Derb(\cor_{\fY}).
\eneq
\end{proposition}
\

\subsection{Convolution and $\gamma$-topology}\label{sect:conv}

In this subsection, we consider a real $n$-dimensional vector space $\BBV$. 
We consider its projective compactification $\BBP=(\BBV\oplus\R\setminus\{0\})/\R^\times$. The pair  $(\BBV,\BBP)$ is a b-analytic manifold and we set
\eq\label{eq:Vi}
&&\Vi=(\BBV,\BBP).
\eneq
If there is no risk of confusion, we simply write $\BBV$ instead of $\Vi$. 

\subsubsection*{Convolution}
We denote  by $s$ the addition map.
\eqn
&&s\cl \BBV\times\BBV\to\BBV,\quad (x,y)\mapsto x+y.
\eneqn
Clearly, $s$  is a morphism of b-analytic manifolds.

We  define the convolution and the non-proper convolution as follows. For 
$F,G\in\Derb_\Rc(\cor_{\Vi})$, we set
\eqn
&&F\star G\eqdot\reim{s}(F\etens G),\quad F\npstar G\eqdot\roim{s}(F\etens G). 
\eneqn
By  Propositions~\ref{pro:opersainftyext} and~\ref{pro:opersainftyint}, both $F\star G$ and $ F\npstar G$ belong to $\Derb_\Rc(\Vi)$. 
One checks easily that both convolution operations are commutative and that usual (proper)  convolution is associative. Note that, denoting by $\BBV_i$ ($i=1,2$) two copies of $\BBV$ one has $F_1\star F_2\simeq (F_1\etens F_2)\conv[12]\cor_{\Gamma_s}$ where $\Gamma_s$ is the graph of $s$ in $\BBV_{12}\times\BBV$. 

\begin{proposition}\label{pro:assocnpstar}
Let $F_i\in\Derb(\Vi)$, $i=1,2,3$. Then
\eqn
&&F_1\npstar F_2\simeq\RD_\BBV( \RD_\BBV F_1\star\RD_\BBV F_2),
\quad(F_1\npstar F_2)\npstar F_3\simeq F_1\npstar( F_2\npstar F_3).
\eneqn
\end{proposition}
\begin{proof}
(i) The first isomorphism follows from Corollary~\ref{cor:opersainftyext}.

\spa
(ii) The second isomorphism follows from the first one and the associativity of the usual convolution.
\end{proof}

\begin{remark}
Proposition~\ref{pro:assocnpstar} is remarkable since, as already mentioned,  the operation $\npstar$ is not associative in general.
\end{remark}

\subsubsection*{$\gamma$-topology}
References to the $\gamma$-topology and its links with sheaf theory are made to~\cites{KS90, KS18}.
We consider a real $n$-dimensional vector space $\BBV$. We set $\dot\BBV=\BBV\setminus\{0\}$ and we recall that $\Vi$ is defined in~\eqref{eq:Vi}. Clearly,  the antipodal map $a\cl\BBV\to\BBV$, $x\mapsto -x$, is a morphism of b-analytic manifolds. 
For a subset $A$ of $\BBV$, we denote by $A^a$ its image by the antipodal map.

A subset $\gamma$ of $\BBV$ is called a cone if  $\R_{>0}\gamma=\gamma$. A closed convex cone  $\gamma$ is proper
if $\gamma\cap \gamma^a=\{0\}$. 

 We consider a cone $\gamma\subset\BBV$  and we assume:
 \eq\label{hyp1}
&&\parbox{75ex}{\em{
$\gamma$ is  a closed  convex proper subanalytic cone with non-empty interior. }
}\eneq
\begin{lemma}\label{lem:gammainfty}
Let $\gamma\subset\dV$ be a cone,  subanalytic in $\dV$. Then  $\gamma$ is subanalytic up to infinity.
\end{lemma}
\begin{proof}
(a) The set $\gamma$ is subanalytic in $\BBV$ by ~\cite{KS90}*{Prop.~8.3.8~(i)}.

\spa
(b) Choose a subanalytic norm $\vvert\cdot\vvert$ on $\BBV$ and consider the real analytic isomorphism  $f\cl \dot\BBV\to\dot\BBV$, $f(x)=x/\vvert x\vvert^2$. The map $f$ defines an automorphism of the b-analytic manifold $\Vi$. 
It is thus enough to check that $f(\gamma)$ is subanalytic in $\BBV$. Since this set is a subanalytic cone, this follows from (a).
\end{proof}
The family of $\gamma$-invariant open subsets $U$ of $\BBV$ (that is, satisfying $U=U+\gamma$) defines a topology, which is called  the {\em $\gamma$-topology} on $\BBV$.
One denotes by $\BBV_\gamma$ the space $\BBV$ endowed with the $\gamma$-topology and one denotes by
\begin{equation}
\phig \cl\BBV \to \BBV_\gamma
\end{equation}
the continuous map associated with the identity.
Note that the closed sets for this topology are the $\gamma^a$-invariant closed subsets of $\BBV$ and that a subset  is $\gamma$-locally closed if it is the intersection of a $\gamma$-closed subset and a $\gamma$-open subset. 

\begin{lemma}\label{le:gamloccld}
Let $A\subset\BBV$. The conditions below are equivalent:
\banum
\item
$A=(U+\gamma)\cap(\ol{U+\gamma^a})$ with  $U$ open and  subanalytic up to infinity. 
\item
$A$ is the intersection of a $\gamma$-closed subset $S$ and a $\gamma$-open subset $U$, both $S$ and $U$ being subanalytic up to infinity.
\item
$A$ is $\gamma$-locally closed and $A$ is subanalytic up to infinity.
\eanum
\end{lemma}
\begin{proof}
(a)$\Rightarrow$(b). It is enough to check that $U$ being subanalytic up to infinity, $U+\gamma$ is subanalytic up to infinity. This set is the image of the set $U\times\gamma$ by the map 
$s\cl\BBV\times\BBV\to\BBV$, $(x,y)\mapsto x+y$. 
Hence, the result follows from Proposition~\ref{pro:opersainfty}.

\spa
(b)$\Rightarrow$(c) is obvious.

\spa
(c)$\Rightarrow$(a). By~\cite{KS18}*{Prop.~3.4}, we may write $A=(U+\gamma)\cap(\ol{U+\gamma^a})$ with  $U=\Int(A)$. 
Therefore, $U$ is subanalytic up to infinity.
\end{proof}

\begin{definition}\label{def:loclo}
Let $A$ be a subset of $\BBV$. 
One says that $A$ is   {\em b-subanalytic  $\gamma$-locally closed} if $A$ satisfies one of the equivalent conditions in Lemma~\ref{le:gamloccld}.
\end{definition}

Let $\gamma$  be a  cone satisfying~\eqref{hyp1}. Recall that one denotes by  $\gamma^\circ\subset \BBV^*$ the polar cone.:
\eqn
&&\gamma^\circ=\{y\in\BBV^*;\langle x, y \rangle\geq0\mbox{ for all }x\in\gamma\}.
\eneqn

\subsubsection*{$\gamma$-constructible sheaves}\label{subsect:gammashv}

Consider the full triangulated subcategories of the category $\Derb(\cor_\BBV)$:
\eq\label{not:0}
&&\left\{\begin{array}{l}
\Derb_{\gammac}(\cor_{\BBV})\eqdot\{F\in\Derb(\cor_{\BBV}) ;\musupp(F)\subset \BBV\times\gammac\},\\[1ex]
\Derb_{\rcg}(\cor_{\Vi})\eqdot\Derb_{\Rc}(\cor_{\Vi})\cap\Derb_{\gammac}(\cor_{\BBV}).
\end{array}\right.
\eneq
We call an object of the category $\Derb_{\rcg}(\cor_{\Vi})$ a $\gamma$-constructible sheaf. 

\begin{theorem}\label{th:rcg}
Let $F\in\Derb_{\rcg}(\cor_{\Vi})$. Then  there exists a finite partition $\BBV=\bigsqcup_{a\in A}Z_a$ where the $Z_a$'s are 
b-subanalytic $\gamma$-locally closed and $F\vert_{Z_a}$ is constant. 
\end{theorem}
\begin{proof}
 This result is proved  by Ezra Miller in~\cite{Mi20b}, using the tools of~\cite{Mi20a}. 
If we make the extra hypothesis  that $F$ is $\PL$ (piecewise linear) and the cone $\gamma$ is polyhedral, then this  result is proved 
in~\cite{KS18}*{Th.~3.18}. Note that in loc.\ cit.\ the notion of being subanalytic up to infinity is not used and the partition (which is called a stratification there) is only locally finite. However,  in our situation, the fact that the partition is finite is implicit in  the first part of the proof. 
\end{proof}

\begin{lemma}\label{le:projgamma}
The endofunctor $\cor_{\gamma^a}\npstar$ of   \,$\Derb(\cor_{\BBV})$ defines a projector 
$\Derb_\Rc(\cor_{\Vi})\to \Derb_{\rcg}(\cor_{\Vi})$. 
\end{lemma}
 Denoting by $\iota$ the embedding $ \Derb_{\rcg}(\cor_{\Vi})\into \Derb_\Rc(\cor_{\Vi})$ and by $p$ the functor 
$\cor_{\gamma^a}\npstar$, we mean that $p\circ\iota$ is an equivalence.
\begin{proof}
We know by~\cite{KS90}*{Prop.~5.2.3} that the functor $\opb{\phig}\roim{\phig}\cl\Derb(\cor_{\BBV})\to \Derb_{\gamma^{\circ a}}(\cor_{\BBV})$ is a projector and we know by~\cite{KS90}*{Prop.~3.5.4} that the two functors $\opb{\phig}\roim{\phig}$ and 
$\cor_{\gamma^a}\npstar$ are isomorphic. Moreover, the functor $\cor_{\gamma^a}\npstar$ sends $\Derb_\Rc(\cor_{\Vi})$ to itself by 
Proposition~\ref{pro:opersainftyext}. 
\end{proof}

\begin{remark}
In general, non proper convolution is not defined on $\Derb_\Rc(\cor_\BBV)$ and, in particular, 
even if $\gamma$ is subanalytic, the functor $\cor_{\gamma^a}\npstar$ does not send $\Derb_\Rc(\cor_{\BBV})$ to itself.
\end{remark}

\section{A short review on constructible functions}\label{section:cstfct}
From now on and until the end of this paper, we assume that $\cor$  is a field of characteristic $0$.

In this section, we recall without proofs the main constructions and results on constructible
functions. References are made to~\cite{Sc91} and~\cite{ KS90}*{\S~9.7}. 

\subsection{From constructible sheaves to constructible functions}

\begin{definition}\label{def:constfct1}
 Let $X$ be a real analytic manifold. 
A function $\vphi \cl X\to \Z$ is constructible if:
\begin{itemize}
\item [(i)] for all $m \in \Z, \opb{\vphi}(m)$ is subanalytic in $X$,
\item [(ii)] the family $\{\opb{\vphi}(m)\}_{m\in \Z}$ is locally finite.
\end{itemize}
\end{definition}
\begin{notation}
For a locally closed subanalytic subset $S\subset X$, we denote by $\un_S$ the characteristic function of $S$ (with values $1$ on $S$ and $0$ elsewhere). 
For $a\in X$ we also set $\delta_a=\un_{\{a\}}$.
\end{notation}

The next result is well known. Note that the implication (b)$\Rightarrow$(d) follows from the  triangulation theorem for compact  subanalytic subsets (see~\cite{Ha76}).
\begin{lemma}\label{lem:3.3}
Let $\vphi$ be a $\Z$-valued function on $X$. The conditions below are equivalent.
\banum
\item
$\vphi$ is constructible,
\item
there exist a locally finite family of  subanalytic locally closed subsets
$\{Z_i\}_{i\in I}$  and $c_i \in \Z$ such that $ \vphi = \sum_ic_i {\bf 1}_{Z_i}$,
\item
there exist a subanalytic stratification 
$\{Z_i\}_{i\in I}$  and $c_i \in \Z$ such that $ \vphi = \sum_i c_i {\bf 1}_{Z_i}$,
\item
same as {\rm (b)} assuming moreover each $Z_i$ compact and contractible.
\eanum
\end{lemma}

\begin{notation}
One denotes by  $\CF(X)$ the group of constructible functions on $X$ and by $\CF_X$ the
presheaf $U\mapsto \CF(U)$. 
\end{notation}
\begin{proposition}
The presheaf  $\CF_X$ is a sheaf on $X$.
\end{proposition}
\begin{proof}
(i) Clearly, the presheaf $U\mapsto\CF(U)$ is separated. 

\spa
(ii) Let $X= \bigcup_{a\in A}U_a$  be an open covering of $X$ and let $\vphi$ be a $\Z$-valued function on $X$ such that $\phi\vert_{U_a}$ is constructible on $U_a$.  Since $X$ is paracompact, one may assume that the covering is locally finite.
For $m\in\Z$,  set $Z_m\eqdot\opb{\vphi}(m)$ and $Z_{m,a}= Z_m\cap U_a$. Each $Z_{m,a}$  is subanalytic in $U_a$,
which implies that $Z_m$  is subanalytic in $X$. 
Moreover, the family $\{Z_{m,a}\}_m$ being locally finite in $U_a$, the family $\{Z_m\}_m$ is locally finite in $X$. Hence, $\phi$ is constructible on $X$. The same argument holds when replacing $X$ with an open subset $U\subset X$.
\end{proof}

 Recall now that if $V\in\Derb(\cor)$ has the property that all its cohomology objects are finite dimensional, one defines its 
 Euler-Poincar{\'e} index by
 \eq\label{eq:EPI1}
 &&\chi(V)=\sum_i(-1)^i\dim H^i(V).
 \eneq
For a constructible sheaf $F$, one defines its local  Euler-Poincar{\'e} index at $x\in X$ by 
\eqn
&& \chi_\loc(F)(x)=\sum_i(-1)^i\dim H^i(F_x).
\eneqn 
Clearly, the function $x\mapsto  \chi_\loc(F)(x)$ is constructible and we get a map: 
\eq\label{eq:localEP}
&&\chi_\loc\cl \Ob(\Derb_\Rc(\cor_X))\to\CF(X).
\eneq

Denote by $\BBK(\shc)$ the Grothendieck group of either an abelian or a triangulated category $\shc$,
and recall that if $\shc$ is abelian then $\BBK(\shc)\isoto\BBK(\Derb(\shc))$. Recall that
if $F\cl\shc\to\shc'$ is a triangulated functor (of triangulated categories), then it defines a linear map $\BBK(\shc)\to\BBK(\shc')$.

In the sequel, we set for short
\eqn
&&\BBK_\Rc(\cor_X)\eqdot\BBK(\Derb_\Rc(\cor_X)).
\eneqn 
The tensor product on $\Derb_\Rc(\cor_X)$ defines a ring structure on $\BBK_\Rc(\cor_X)$, with unit the image of the constant sheaf $\cor_X$.
The next  theorem clarifies the notion of constructible function. 

\begin{theorem}[\cite{KS90}*{Th.~9.7.1}]\label{th:Grgroup1}
Let $X$ be a  real analytic manifold. Then the map $\chi_\loc$ defines an isomorphism of commutative unital algebras  \lp we keep the same notation\rp\,
$\chi_\loc\cl \BBK_\Rc(\cor_X)\isoto\CF(X)$.
\end{theorem}
Note that if $\chi_\loc(F)=\phi$ and $S\eqdot\supp(\phi)$, then  $\chi_\loc(F_S)=\phi$. Hence, given $\phi\in\CF(X)$, we may always represent $\phi$ with a constructible sheaf of same support. 
We have the general ``principle'' that we shall make explicit in the sequel:
\eqn
&&\parbox{80ex}{
{\em The operations on constructible functions are the image by the local Euler-Poincar{\'e} index $\chi_\loc$ of the corresponding operations on constructible sheaves.}
}\eneqn
In the sequel, we shall also encounter the global  Euler-Poincar{\'e} indices  of a sheaf $F$ (assuming that these indices are finite):
\eq\label{eq:globalEP}
&&\chi(F)=\chi(\rsect(X;F)),\quad \chi_c(F)=\chi(\rsect_c(X;F)).
\eneq
 In particular, for  a locally closed subanalytic subset $Z$ of $X$, we set
\eq\label{eq:globalEP2}
&&\chi(\cor_Z)=\chi(\rsect(Z;\cor_Z)),\quad \chi_c(\cor_Z)=\chi(\rsect_c(Z;\cor_Z)).
\eneq
 Classically,  the Euler-Poincar{\'e} index of a compact subanalytic set $K$  is defined by
\eq\label{eq:clEPI}
&&\chi(K)=\chi(\Q_K).
\eneq

Recall that, denoting by $j\cl Z\into X$ the embedding, $\cor_{XZ}=\eim{j}\cor_Z$. Hence, $\rsect_c(Z;\cor_Z)\simeq\rsect_c(X;\cor_{XZ})$ and if $Z$ is closed,
$\rsect(Z;\cor_Z)\simeq\rsect(X;\cor_{XZ})$  since  $\cor_{XZ}\simeq\oim{j}\cor_Z$ in this case. However,  
$\rsect(Z;\cor_Z)\simeq\rsect(X;\roim{j}\cor_{Z})\neq\rsect(X;\cor_{XZ})$ in general.

 \begin{remark}\label{rem:car0}
Recall that $\cor$ be a field of characteristic $0$. Let $Z$ be a locally closed  subanalytic subset of $X$. Applying the projection formula, we get the isomorphism 
$\rsect_c(Z;\Q_Z)\tens\cor\isoto \rsect_c(Z;\cor_Z)$. Hence
\eq\label{eq:car01}
&&\chi_c(\cor_Z)=\chi_c(\Q_Z).
\eneq
\end{remark}

\subsection{Operations}
\subsubsection*{Internal operations}
The sum on $\CF(X)$ is the image by $\chi_\loc$ of the direct sum for sheaves, the unit $\un_X$   is the image of the constant sheaf $\cor_X$, the map $\phi\mapsto -\phi$ corresponds to the shift $F\mapsto F\,[+1]$ and the usual product on $\CF(X)$ is the image of the tensor product.

\subsubsection*{External product}
 For two real analytic manifolds $X$ and $Y$, one defines the morphism
 \eq\label{eq:exterfct}
 &&\etens\cl \CF_X \etens \CF_Y \to \CF_{X\times Y},\quad  (\vphi \etens \psi) (x,y) = \vphi(x)\psi(y).
 \eneq
 
 \subsubsection*{Inverse image or composition}
Let $f \cl X\to Y$ be a morphism of  real analytic manifolds. One defines the inverse image morphism
\eq\label{eq:inversfct}
&& f^*\cl  \opb{f}\CF_Y \to \CF_X,\quad (f^*\psi)(x) = \psi(f(x))\mbox{ for }\psi\in  \CF(Y).
\eneq
(Recall that a morphism $\opb{f}\CF_Y \to \CF_X$ is nothing but a morphism $\CF_Y \to \oim{f}\CF_X$.)

 Inverse images are functorial, that is, if $f\cl X\to Y$ and $g\cl Y \to Z$ are
morphisms of manifolds, then: 
\eqn
&&f^*\circ g^* = (g\circ f)^*.
\eneqn

 \subsubsection*{Direct image or integral}
Recall that, if $K$ is a subanalytic compact subset of $X$, then the   Euler-Poincar\'e index $\chi(K)$ is 
defined in~\eqref {eq:clEPI}.
In particular, if $K$ is contractible, then $\chi(\cor_K)=1$ and one sets in this case
\eq\label{eq:chiK}
&&\int_X\mbox{\bf 1}_K =1.
\eneq
If $\vphi$ has compact support, one may assume that the sum in Lemma~\ref{lem:3.3}~(d)
 is finite, and one checks (using either Theorem~\ref{th:Grgroup1} or the triangulation theorem for subanalytic sets) that the
integer $\sum_i c_i$ depends only on $\vphi$, not on its decomposition. One sets:
\eqn
&& \int_X \vphi = \sum_i c_i.
\eneqn
In particular, if $Z$ is locally closed relatively compact and subanalytic in $X$, then (see~\eqref{eq:globalEP2}):
\eq\label{eq:EPI3}
&&\int_X\un_Z=\chi_c(\cor_Z).
\eneq

By~\eqref{eq:car01}, this integer does not depend on the choice of $\cor$ as soon as $\cor$ has characteristic $0$.

One  should be aware that the integral is not positive, that is
\eqn
&&\vphi\geq0 \mbox{ does not imply } \int_X\vphi\geq0.
\eneqn
For example, take $X=\R$ and $\vphi={\bf 1}_{(-1,1)}$. Hence, $\vphi\geq0$ and $\int_\R\vphi=-1$. 

 Let $f \cl X\to Y$ be a morphism of  real analytic manifolds. One defines the direct image morphism
 \eq\label{eq:intcst}
 &&\int_X\cl \eim{f}\CF_X\to\CF_Y,\quad \bl\int_f \vphi\br(y) = \int_X{\bf 1}_{\opb{f}(y)}\cdot\phi.
 \eneq
Recall that a section of $\eim{f} \CF_X$ on an open subset $V\subset Y$ is a section of $\CF_X(\opb{f}V)$ such that $f$ is proper on its support. Hence the integral makes sense as a function but it is not obvious that it is a constructible function. 
This follows for example from the corresponding result for direct images of constructible sheaves. Indeed, let $F\in\Derb_\Rc(\cor_X)$ be such that
$\chi_\loc(F)=\phi$ and $\supp(F)=\supp(\phi)$. Then $\int_f \vphi=\chi_\loc(\reim{f}F)$. 

Direct images are functorial, that is, if $f\cl X\to Y$ and $g\cl Y \to Z$ are
morphisms of manifolds, then: 
\eqn
&&\int_g\conv\int_f=\int_{g\conv f}.
\eneqn

 \subsubsection*{Duality}
On $X$, the dual  of a constructible function is the image by $\chi_\loc$ of the duality functor $\RD_X$ for sheaves. 
For $F\in\Derb(\cor_X)$ and $x_0\in X$, one has
\eqn
&&(\RD_X F)_{x_0}\simeq (\rsect_{\{x_0\}}(F))^*,
\eneqn
 where ${}^*$ denotes the  duality functor for  $\cor$-vector spaces. Since $F$ is constructible, there exists a local chart and $\epsilon_0>0$ such that,
denoting by  $B_\epsilon(x_0)$ the open ball with center $x_0$ and radius $\epsilon > 0$ in this chart, one has for $0<\epsilon\leq\epsilon_0$:
\eqn
&&\rsect_{\{x_0\}}(F)\simeq\rsect_c(B_\epsilon(x_0);F)\simeq \reim{a_X}(F\tens\cor_{B_\epsilon(x_0)}).
\eneqn
Hence, one defines the dual of a constructible function $\vphi$ on $X$ as follows. Let $x_0 \in X$, and
choose a local chart in a neighborhood of $x_0$ and $\epsilon>0$ as above. One sets
\eq\label{eq:dual}
&& (\RD_X\vphi) (x_0)=\int_X\vphi\cdot {\bf 1}_{B_\epsilon(x_0)}.
\eneq
The integral  $\int_X\vphi\cdot {\bf 1}_{B_\epsilon(x_0)}$ neither depends on the local chart nor on $\epsilon$, for
$0<\epsilon \leq\epsilon_0$, for some $\epsilon_0>0$ depending on $x_0$. 

We get a morphism of sheaves $\RD_X\cl \CF_X\to\CF_X$ and this morphism is  an involution, that is, 
 \eqn
 &&\RD_X\circ \RD_X \simeq \id_X. 
 \eneqn
Moreover, duality  commutes with integration. Assuming that $f$ is proper on the support of $\phi$, one has:
 \eq\label{eq:dualint}
&&\RD_Y(\int_f\vphi) = \int_f \RD_X(\vphi).
\eneq
 By mimicking a classical formula for constructible sheaves, one sets
 \eq\label{eq:hom cstfct}
 &&hom(\phi,\psi)\eqdot\RD_X(\RD_X\psi\cdot\phi).
 \eneq
\begin{example}\label{exa:dualA}
Let $Z$ be a  closed subanalytic subset of $X$ and assume that $Z$ is a $C^0$-manifold of dimension $d$ with boundary 
$\partial Z$. Set $A=Z\setminus\partial Z$. Hence, locally on $X$,  $Z\subset X$ is topologically isomorphic to $\ol U\subset\R^n$ where $U$ is a convex open subset of $\R^d\subset\R^n$ and $A\simeq U$. We thus have
\eq\label{eq:dualofone}
&&\RD_X\un_Z=(-1)^d\un_{A}
\eneq
Moreover
\eqn
&&\int_X\un_{\partial Z}=\int_X\un_{Z}-\int_X\un_A=(1-(-1)^d)\int_X\un_{Z}.
\eneqn
When $Z$ is a closed convex polyhedron, one recovers the classical Euler formula.
\end{example}

\subsubsection*{Other operations}
In fact, most (if not all) operations on constructible sheaves admit a counterpart in the language of constructible functions. In~\cite{KS90}*{Def.~9.7.8} one defines the specialization  $\nu_M$ along a submanifold $M$, its Fourier-Sato transform, the microlocalization $\mu_M$ and  $\muhom$:
 \eqn
 &&\nu_M\cl \CF(X)\to\CFR(T_MX), \quad\mu_M\cl \CF(X)\to\CFR(T^*_MX)\\
&&\hspace{12ex} \muhom\cl \CF(X)\times\CF(X)\to\CFR(T^*X). 
\eneqn
Here, for a vector bundle $E\to M$, one denotes by $\CFR(E)$ the subspace of $\CF(E)$ consisting of functions  constant on the orbits of the $\R^+$-action.

One can also define the micro-support of $\phi\in\CF(X)$ by setting
\eq
&&\SSi(\phi)=\supp(\muhom(\phi,\phi)).
\eneq

\section{Constructible functions  up to infinity}\label{section:fctinfty}

\subsection{Definitions}

\begin{definition}\label{def:constinfty}
Let $\fX=(X,\bX)$ be a b-analytic manifold.
\banum
\item
A function $\vphi \cl X\to \Z$ is constructible up to infinity, or b-constructible for short,  if:
\begin{itemize}
\item [(i)] for all $m \in \Z, \opb{\vphi}(m)$ is subanalytic up to infinity,
\item [(ii)] the family $\{\opb{\vphi}(m)\}_{m\in \Z}$ is  finite.
\end{itemize}
\item
We denote by $\CF(\fX)$ the space of functions on $X$ constructible up to infinity.
\item
For any function $\phi$ on $X$,  we denote by  $\eim{j_X}\phi$ the function on $\bX$ obtained as  the function $\phi$ on $X$ extended  by $0$ on $\bX\setminus X$.
\eanum
\end{definition}

\begin{lemma}\label{lem:cstonVi}
Let $\vphi\in\CF(X)$. The conditions below are equivalent.
\banum
\item
The function $\vphi$ is constructible up to infinity,
\item
The function $\eim{j_X}\phi$  belongs to $\CF(\bX)$.
\item
There exists $\psi\in\CF(\bX)$ such that $\phi=\psi\vert_X$. 
\item
There exist a  finite family of   locally closed b-subanalytic  subsets 
$\{Z_i\}_{i\in I}$  and $c_i \in \Z$ such that $ \vphi = \sum_i c_i {\bf 1}_{Z_i}$.
\eanum
\end{lemma}
\begin{proof}
(a)$\Rightarrow$(b).   By the hypothesis, one may write $\phi=\sum_ic_i\un_{Z_i}$ where the sum is finite and the $Z_i$'s are 
 subanalytic up to infinity. Therefore, $\un_{Z_i}\in\CF(\bX)$ and the result follows from Lemma~\ref{lem:3.3}.
  
\spa
(b)$\Rightarrow$(c) is obvious. 

\spa
(c)$\Rightarrow$(d) and (c)$\Rightarrow$(a).
By definition, for each $m\in\Z$, $Z_m\eqdot\opb{\psi}(m)$ is subanalytic in $\bX$ and the family $\{\Z_m\}_m$ is locally finite. 
Therefore, $Z_m\cap X$ is subanalytic in $X$ and $X$ being relatively compact, the family  $\{X\cap Z_m\}_m$ is finite.

\spa
(d)$\Rightarrow$(b) is obvious.
\end{proof}

Clearly, $\CF(\Xinf)$ is a subalgebra of $\CF(X)$. 

Let us denote by $\CF_\fX$ the presheaf on $\Xsai$ given by $U\mapsto \CF(U_\infty)$.
\begin{proposition}
The presheaf $\CF_\fX$ is a sheaf on $\Xsai$.
\end{proposition}
The proof is straightforward.

Recall Theorem~\ref{th:Grgroup1} and denote now by  $\BBK_\Rc(\cor_{\fX})$ the Grothendieck group of the category $\Derb_\Rc(\cor_{\fX})$.

\begin{theorem}\label{th:Grgroup2}
The isomorphism of commutative unital algebras $\chi_\loc\cl\BBK_\Rc(\cor_X)\isoto\CF(X)$ induces an isomorphism
$\chi_\loc\cl \BBK_\Rc(\cor_{\fX})\isoto\CF(\fX)$.
\end{theorem}
\begin{proof}
(i) The map $\chi_\loc$ takes its values in $\CF(\fX)$. Indeed, for $F\in \Derb_\Rc(\cor_{\fX})$, $\chi_\loc(F)=j_X^*(\chi_\loc(\eim{j_X}F))$. 

\spa
(ii) The map $\chi_\loc\cl \BBK_\Rc(\cor_{\fX})\to\CF(\fX)$ is injective by the same arguments as in the proof 
of~\cite{KS90}*{Th.~9.7.1}.

\spa
(iii) The map $\chi_\loc$ is surjective since for $Z$ locally closed and subanalytic up to infinity, $\un_Z=\chi_\loc(\cor_Z)$ and $\cor_Z$ is constructible up to infinity.
\end{proof}

\subsection{Operations}\label{section:cftinfty}

\begin{lemma}
If $\phi\in\CF(\fX)$, then $\RD_X\phi\in\CF(\fX)$.
\end{lemma}
\begin{proof}
The result follows from Lemma~\ref{lem:cstonVi}~(c) since duality commutes with restriction to an open subset. 
\end{proof}

Let  $\phi\in\CF(\fX)$. 
One sets 
\eq
&&\oim{j_X}\phi=\RD_{\bX}\eim{j_X}\RD_X\phi.\label{eq:oimj}
\eneq

The next result follows from the corresponding result for sheaves.
\begin{lemma}\label{le:eimjoimj}
If $\phi\in\CF(\fX)$ has compact support in $X$, then $\oim{j_X}\phi=\eim{j_X}\phi$.
\end{lemma}

\begin{proposition}\label{pro:opercfsainfty}
Let $\fX$ and $\fY$ be two b-analytic manifolds.
\banum
\item
Let  $\vphi\in\CF(\Xinf)$ and  $\psi\in\CF(\Yinf)$. Then the function $\vphi\etens\psi$, defined by  $(\vphi\etens\psi)(x,y)=\phi(x)\psi(y)$, belongs to $\CF((X\times Y)_\infty)$. 
\item
Let $f\cl \fX\to \fY$ be a morphism of b-analytic manifolds and let 
 $\psi\in\CF(\Yinf)$. Then the function $f^*\psi$ defined by $f^*\psi(x)=\psi(f(x))$ belongs to $\CF(\Xinf)$.
 \eanum
 \end{proposition}
 In other words we have extended the morphisms~\eqref{eq:inversfct} and~\eqref{eq:exterfct} to b-analytic manifolds. 
 \begin{proof}
 (a) Apply  condition (b) of Lemma~\ref{lem:cstonVi}.
 
 \spa
 (b)  Apply   Proposition~\ref{pro:opersainfty} together with Definition~\ref{def:constinfty}.
\end{proof}
Although we shall not use it, let us mention that one can also define the internal hom and  the exceptional inverse image  by the formulas
 \eq\label{eq:epbpsi}
&&\ba{l}
 hom(\phi,\psi)\eqdot\RD_X(\RD_X\psi\cdot\phi),\quad \phi,\psi\in\CF(\fX),\\
\epb{f}\psi\eqdot \RD_X f^*(\RD_Y\psi), \psi\in\CF(\fY).
 \ea
 \eneq
 
 Now we study the integrals of constructible functions up to infinity.
 One can define two integrals of $\phi\in\CF(\fX)$. One sets
\eq\label{eq:intnp}
&&\int_X\phi\eqdot\int_{\bX}\eim{j_X}\phi,\quad \int^\np_X\phi\eqdot \int_{\bX}\oim{j_X}\phi.
\eneq
Recall notations~\eqref{eq:globalEP2}.

\begin{lemma}\label{le:intnpdual}
\banum
\item
One has $\int^\np_X\phi=\int_X\RD_X\phi$.
\item
Let $Z$ be a locally closed  b-subanalytic subset of $X$. Then
\footnote{In~\cite{Sc20}*{v3},  it was written  $\int^\np_X\un_Z=\chi(Z)$, which is  not correct.}
$\int_X\un_Z=\chi_c(\cor_Z)$.
\item
The integrals $\int_X\phi$  and  $\int^\np_X\phi$  do not depend on the choice of $\bX$. 
\eanum 
\end{lemma}
\begin{proof}
(a)   follows from 
\eqn
&&\int_{\bX}\oim{j_X}\phi=\int_{\bX}\RD_{\bX}\eim{j_X}\RD_X\phi=\int_{\bX}\eim{j_X}\RD_X\phi=\int_X\RD_X\phi.
\eneqn
where the second equality follows from~\eqref{eq:dualint} applied with $Y=\rmpt$.

\spa
(b) Recall~\eqref{eq:EPI3}. Also recall that $a_{Z}$ is the map $Z\to\rmpt$ and similarly with $a_{\bX}$. Denoting by $j_Z$ the embedding $Z\into\bX$,
we have
 \eqn\label{eq:chic}
 &&\int_X\un_Z=\chi(\reim{a_{\bX}}\reim{j_X}\cor_{XZ})=\chi(\reim{a_{\bX}}\reim{j_Z}\cor_{Z})=\chi(\reim{a_Z}\cor_Z)=\chi_c(\cor_Z). 
\eneqn

\spa
(c) follows from (b)  and (a).
\end{proof}

 \begin{example}
Let $X=\R$. Then:
\eqn
&&\parbox{75ex}{
(i) One has  $\int_\R\un_\R=-1$, $\int^\np_\R\un_\R=1$.\\

\vspace{-2ex}
(ii) Let $U=(-\infty,b)$ with $-\infty<b<\infty$. Then $ \int_\R\un_{U}=-1$, $ \int^\np_\R\un_{U}=0$.\\

\vspace{-2ex}
(iii) Let $Z=(-\infty,b]$ with  $-\infty<b<+\infty$.  Then $\int_\R\un_Z=0$, $\int^\np_\R\un_Z=1$.\\

\vspace{-2ex}
(iv) Let $S=[a,b]$ with $-\infty< a\leq b<+\infty$. Then $\int_\R\un_S=\int^\np_\R\un_S=1$.\\

\vspace{-2ex}
(v) Let $Z=(a,b)$ with  $-\infty<a< b<+\infty$.  Then $\int_\R\un_Z= \int^\np_\R\un_Z=-1$.\\

\vspace{-2ex}
(vi) Let $Z=[a,b)$ with  $-\infty<a\leq b<+\infty$.  Then $\int_\R\un_Z= \int^\np_\R\un_Z=0$.
}
\eneqn
Indeed, (i) is obvious. 
 Let $U$ be as in (ii). Then $U$ is  topologically isomorphic to $\R$ and we get   $ \int_\R\un_{U}=-1$. By the additivity of the integral, we deduce that for $Z$ as in (iii),  $\int_\R\un_Z=0$. By Lemma~\ref{le:intnpdual}~(a),  we get 
  $ \int^\np_\R\un_{U}=0$ and by additivity, $\int^\np_\R\un_Z=1$.
Finally, (iv), (v) and (vi) are obvious. 
  \end{example} 
Let $f\cl \fX\to \fY$ be  a morphism  of b-analytic manifolds and let $\phi\in\CF(\fX)$. Similarly as in~\eqref{eq:intcst}, one sets for $y\in Y$:
\eq\label{eq:integrinftyf}
&&(\int_f\phi)(y)=\int_X\un_{\opb{f}(y)}\cdot\phi.
\eneq
Of course, when $Y=\rmpt$, one recovers~\eqref{eq:intnp}.

\begin{lemma}
The function $\int_f\phi$ defined by~\eqref{eq:integrinftyf} belongs to $\CF(\fY)$.
\end{lemma}
\begin{proof}
Let us choose $F\in\Derb(\cor_\fX)$ such that $\chi_\loc(F)=\phi$. Then $(\int_f\phi)(y)=\chi_\loc(\reim{f}F)$ and $\reim{f}F\in\Derb(\cor_\fY)$. 
\end{proof}
Hence, we have constructed a morphism
\eqn
&&\int_f\cl \oim{f}\CF_\fX\to\CF_\fY,\quad \phi\mapsto\int_f\phi.
\eneqn
We also define 
\eq\label{eq:integrinftyf2}
&&\int^\np_f\cl \oim{f}\CF_\fX\to\CF_\fY,\quad \int^\np_f\phi\eqdot \RD_Y\int_f\RD_X\phi.
\eneq

The next results are easily checked.
\begin{itemize}
\item
 If $f$ is proper on $\supp(\phi)$, then $\int_f\phi=\int^\np_f\phi$.
 \item
If $\phi=\chi_\loc(F)$ for some  $F\in\Derb_\Rc(\cor_{\fX})$, then $\int_f\phi=\chi_\loc(\reim{f}F)$ and 
$\int^\np_f\phi=\chi_\loc(\roim{f}F)$.
\item
Let  $g\cl \fY\to \fZ$  be another morphism of b-analytic manifolds. Then
\eqn
&&\int_{g\conv f}\phi=\int_g\int_f\phi,\quad \int^\np_{g\conv f}\phi=\int^\np_g\int^\np_f\phi.
\eneqn
\end{itemize}

\subsubsection*{Base change formula and  projection formula}

\begin{proposition}\label{pro:basechinfty}
Consider the Cartesian square~\eqref{diag:cart1} and let $\vphi\in\CF(\fX)$. Then $ \int_{f'} (g^{\prime * }\vphi)$ is well defined,  belongs to $\CF(\fY)$ and 
\eq\label{eq:basechform2}
&&g^* \int_f \vphi = \int_{f'} (g^{\prime * }\vphi).
\eneq
\end{proposition}
\begin{proof}
Choose $F\in\Derb_\Rc(\cor_{\fX})$ such that $\chi_\loc(F)=\eim{j_X}\phi$.
Then apply the base change formula for sheaves (Proposition~\ref{pro:basechinftyshv}).
\end{proof}

\begin{proposition}\label{pro:projforminfty}
Let $f\cl \fX\to \fY$ be a morphism of b-analytic manifolds, let  $\vphi\in\CF(\fX)$ and $\psi\in\CF(\fY)$. Then
\eq\label{eq:projforminfty2}
&&\int_f(\vphi\cdot f^*\psi)=\psi\int_f\vphi.
\eneq
\end{proposition}
\begin{proof}
Choose $F\in\Derb_\Rc(\cor_{\fX})$ such that $\chi_\loc(F)=\eim{j_X}\phi$ 
and  choose $G\in\Derb_\Rc(\cor_{\fY})$ such that $\chi_\loc(G)=\eim{j_Y}\psi$.
Then apply the projection formula for sheaves (Proposition~\ref{pro:projforminftyshv})..
\end{proof}
\begin{example}
Equality~\eqref{eq:projforminfty2} is no longer true when replacing $\int_f$ with $\int_f^\np$. Set $X=\R^2$ with coordinates $(y,t)$ and $Y=\R$, $f$ being the first projection. Let $\phi=\un_S$ with $S=\{(y,t);t=1/(1-y^2), -1<y<1\}$ and let $\psi=\un_Z$ with $Z=(-1,1)$. 
One checks easily that $\phi$ is subanalytic up to infinity when choosing for example for $\bX$ the projective compactification of $\R^2$. 
We have $\un_S\cdot f^*\un_Z=\un_S$,   $\int_f\un_S=\un_Z$ and $\RD_X\un_S=-\un_S$ (see Example~\ref{exa:dualA}). Hence, 
\eqn
&&\int_f^\np\un_S\cdot f^*\un_Z=\int^\np_f\un_S=\RD_Y\int_f\RD_X\un_S=-\RD_Y\un_Z=\un_{[-1,1]},\\
&&\un_Z\cdot\int_f^\np\un_S=\un_Z\cdot \un_{[-1,1]}=\un_{(-1,1)}.
\eneqn
\end{example}

\subsubsection*{ Convolution of kernels}
Recall  Diagram~\ref{diag:123} when replacing the manifolds $X_i$ with b-analytic manifolds $X_{i\infty}$ ($i=1,2,3$). 
Let $\lambda_{12}\in \CF(X_{12\infty})$ and $\lambda_{23}\in \CF(X_{23\infty})$. 
It follows from Proposition~\ref{pro:opercfsainfty} that the function
 \eq\label{eq:convfctinfty}
 &&\lambda_{12}\cconv[2]\lambda_{23}\eqdot\int_{q_{13}}q_{12}^*\lambda_{12}\cdot q_{23}^*\lambda_{23}.
\eneq
is well-defined and belongs to $\CF(X_{13\infty})$. 
Moreover 

\begin{theorem}\label{th:assocconvinfty}
Let $\lambda_{ij}\in\CF(X_{ij\infty})$ \lp$i=1,2,3,4$, $j=i+1$\rp.
One has
\eqn
&&(\lambda_{12}\cconv[2]\lambda_{23})\conv[3]\lambda_{34}=\lambda_{12}\cconv[2](\lambda_{23}\conv[3]\lambda_{34})\in\CF(X_{14\infty}).
\eneqn
\end{theorem}
One can prove this theorem by mimicking 
 the classical proof for sheaves, using now Propositions~\ref{pro:basechinfty} and~\ref{pro:projforminfty}. One can also prove this result by replacing each $\lambda_{ij}$ with a kernel $K_{ij}\in\Derb_\Rc(\cor_{X_{ij\infty}})$.

\subsection{$\gamma$-constructible functions}\label{subsection:gammatop}

As already mentioned in the introduction, $\gamma$-constructible functions appear  naturally in TDA (see \cites{CGR12, Leb21, KiM21} among others).

Let $\BBV$ and $\Vi$  be as in \S~\ref{sect:conv}.
We  define the convolution and the non-proper convolution similarly as for sheaves (see Proposition~\ref{pro:assocnpstar}). For 
$\phi,\psi\in\CF(\Vi)$, we set
\eqn
&&\phi\star \psi\eqdot\int_{s}\phi\etens \psi,\quad \phi\npstar \psi\eqdot\int^\np_{s}\phi\etens \psi. 
\eneqn
By  the preceding results, both $\phi\star \psi$ and $ \phi\npstar \psi$ belong to $\CF(\Vi)$. Note that
\eqn
&&\phi\star \psi=\psi\star \phi,\quad \phi\npstar \psi=\psi\npstar \phi.
\eneqn

\begin{lemma}\label{le:assocnpstar}
Let $\phi_i\in\CF(\Vi)$, $i=1,2,3$. Then 
\eqn
&&\phi_1\npstar\phi_2=\RD_X( \RD_X\phi_1\star\RD_X\phi_2),
\quad(\phi_1\npstar\phi_2)\npstar\phi_3=\phi_1\npstar(\phi_2\npstar\phi_3).
\eneqn
\end{lemma}
\begin{proof}
The first equality follows from the definition of $\int^\np$ (see~\eqref{eq:integrinftyf2})
and the second  equality follows from the first one.
\end{proof}

 We consider a cone $\gamma\subset\BBV$  and we assume~\eqref{hyp1}, that is, 
$\gamma$ is  a closed  convex proper subanalytic cone with non-empty interior. 
Recall that $\cor_\gamma$ is then constructible up to infinity.

\begin{definition}\label{def:gammaconst}
Let $\vphi\in\CF(\Vi)$. We say that $\vphi$ is $\gamma$-constructible if  there exists a finite covering $\BBV=\bigcup_aZ_a$ such that $\vphi=\sum_ac_a\un_{Z_a}$ and the $Z_a$'s are b-subanalytic $\gamma$-locally closed  subsets of  $\BBV$. 
 We denote by $\CF(\Vg)$ the space of  $\gamma$-constructible functions on $\BBV$.
\end{definition} 
By construction, we have $\CF(\Vg)\subset\CF(\Vi)$.

Recall notations~\eqref{not:0} and  denote by $\BBK_{\rcg}(\cor_{\Vi})$ the Grothendieck group of the category $\Derb_{\rcg}(\cor_{\Vi})$.

\begin{theorem}\label{th:Grgroup3}
The isomorphism of commutative unital algebras $\chi_\loc\cl\BBK_\Rc(\cor_\BBV)\isoto\CF(\BBV)$ induces an isomorphism
$\chi_\loc\cl\BBK_{\rcg}(\cor_{\Vi})\isoto\CF(\Vg)$.
\end{theorem}
\begin{proof}
(i) It follows from Theorem~\ref{th:rcg} that the map $\chi_\loc$ takes its values in $\CF(\Vg)$. 

\spa
(ii)  The map $\chi_\loc$ is injective by Lemma~\ref{le:projgamma}. Indeed,  if $\sha$ is a full triangulated subcategory of a triangulated category $\sht$ and if there is a projector $P\cl \sht\to\sha$, then $P$ induces a projector $\BBK(P)\cl\BBK(\sht)\to\BBK(\sha)$.
In particular, $\BBK(\sha)$ is a subgroup of $\BBK(\sht)$.

\spa
(iii) The map  $\chi_\loc$ is surjective since for $Z$ subanalytic $\gamma$-locally closed, $\un_Z=\chi_\loc(\cor_Z)$ 
and $\cor_Z\in\Derb_\Rc(\Vi)$. Moreover, $\SSi(\cor_Z)\subset \BBV\times\gamma^{\circ a}$ by~\cite{KS18}*{Cor.~1.8}.
\end{proof}

The projector of Lemma~\ref{le:projgamma} allows us to construct 
 a projector $\CF(\Vi)\to\CF(\Vg)$.

\begin{proposition}\label{pro:constproj}
\banum
\item
Let $\phi\in\CF(\Vi)$. Then $\phi\npstar\un_{\gamma^a}$ belongs to  $\CF(\Vg)$.
\item
If $\phi\in \CF(\Vg)$, then $\phi\npstar\un_{\gamma^a}=\phi$.
\eanum
\end{proposition}
\begin{proof}
The result follows from Theorem~\ref{th:Grgroup3}, Lemma~\ref{le:projgamma} and the fact that the operation $\npstar$ commutes with $\chi_\loc$. 
\end{proof}

\section{Correspondences for constructible functions}\label{section:correspondence}
This section is a variation on~\cite{Sc95} in which we replace some properness hypotheses with that of being constructible 
up to infinity.

\subsection{Correspondences}
Consider the situation of Diagram~\eqref{diag:123} when replacing the manifolds $X_i$ with b-analytic manifolds $X_{i\infty}$ ($i=1,2,3$). 
Assume to be given two locally closed subsets subanalytic up to infinity:
\eqn
&&S_1\subset X_{12},\quad S_2\subset X_{23}.\
\eneqn
We set, for $\phi\in\CF(X_{1\infty})$
\eqn
&&\shr_{S_1}(\phi)=\phi\conv\un_{S_1}=\int_{q_2}q_1^*\phi\cdot\un_{S_1}.
\eneqn
Set
\eq
\lambda&\eqdot&\un_{S_1}\cconv[2]\un_{S_2}\in\CF(X_{13\infty}).
\eneq
Applying Theorem~\ref{th:assocconvinfty}, we get that $\lambda$ is well defined and moreover
\eq\label{eq:phiconvl}
&&\shr_{S_2}\conv \shr_{S_1}(\phi)=\phi\conv\lambda.
\eneq
Now we assume that $X_1=X_3$ and we change our notations, setting
\eqn
&&X_1=X_3=X,\quad X_2=Y.
\eneqn

For $(x,x')\in X\times X$, let
\eq\label{eq:S12xx}
&&S_{12}(x,x')=\{y\in Y; (x,y)\in S_1, (y,x')\in S_2\}=(S_1\times_YS_2)\cap\opb{q_{13}}(x,x').
\eneq
Then 
\eq\label{eq:lambda12xx}
&&\lambda(x,x')=\int_{q_{13}}\un_{S_1\times_YS_2}\cdot\un_{\{\opb{q_{13}}(x,x')\}}=\int_Y\un_{S_{12}(x,x')}.
\eneq

We now consider the hypothesis
\eq\label{eq:H3}
\left\{\parbox{70ex}{
there exists $a, b \in \Z$ such that, for $(x,x')\in X\times X$:\\
$\lambda(x,x')=\left\{\begin{array}{ll}
a& \mbox{ if }x\neq x',\\
b & \mbox{ if }x=x'.
\end{array} \right.$
}\right.\eneq
Writing $\lambda(x,x')=(b-a)\un_\Delta+a\un_{X\times X}$, we get:

\begin{corollary}[{\cite{Sc95}*{Th.~3.1}}]\label{cor:invfor}   
Assume~\eqref{eq:H3}. Let $\phi \in \CF(X)$. Then:
\eqn
&&\shr_{S_2} \circ \shr_{S_1}(\phi) = (b - a)\phi +a\int_X\phi.
\eneqn
\end{corollary}
Here, $a\int_X\phi\in\Z$ is identified with the constant function $(a\int_X\phi)\cdot\un_X$.

\subsubsection*{Application to flag manifolds}
Let $\BBW$ be a real $(n+1)$-dimensional vector space (with $n\geq 2$) and denote by $F_{n+1}(p,q)$, with $1\leq p\leq q \leq n$, the set of pairs 
$\{(l,h)\}$ of linear subspaces of $\BBW$ with $l\subset h$ and $ \dim l = p$, $\dim h=q$. One sets 
$F_{n+1}(p)= F_{n+1}(p,p)$ and denotes as usual by $q_1$ and $q_2$ the two projections defined on 
$F_{n+1}(p)\times F_{n+1}(q)$.
Then $F_{n+1}(p,q)$ is a real compact submanifold of $F_{n+1}(p)\times F_{n+1}(q)$, called the incidence relation.
We denote by
$F_{n+1}(q,p)$ its image by the map  $F_{n+1}(p)\times F_{n+1}(q)\to F_{n+1}(q)\times F_{n+1}(p),
(x,y)\mapsto (y,x)$. In the sequel, we set
\eqn
&&X= F_{n+1}(p),  \quad Y=F_{n+1}(q),\quad S=F_{n+1}(p,q)\subset X\times Y,\quad S'=F_{n+1}(q,p)\subset Y\times X.
\eneqn
Now we shall assume $p=1$ and  $q>1$. Recall that  $F_{n+1}(1) = \BBP_n$, the $n$-dimensional
real projective space. 

In order to apply Corollary~\ref{cor:invfor}, it is enough to calculate $\lambda_{12}(x,x')$ given by~\eqref{eq:lambda12xx}
and~\eqref{eq:S12xx} with $S_1=S$ and $S_2=S'$. 
Set
\eqn
&&\mu_{n+1}(q)=\chi(F_{n+1}(q)).
\eneqn
\begin{proposition}\label{pro:invforA}
Let $\phi \in \CF(\BBP_n)$. Then:
\eqn
&&\shr_{(n+1;q,1)} \circ \shr_{(n+1;1,q)} (\phi) = (\mu_n(q-1) - \mu_{n-1}(q-2))\phi +\mu_{n-1}(q-2)\int_{\BBP_n}\phi. 
\eneqn
\end{proposition}
\begin{proof}
 Let us represent $x$ and $x'$ by lines in $\BBW$ and $y\in F_{n+1}(q)$ by a $q$-dimensional linear subspace. 
Then the set  $S_{12}(x,x')$ is the set of $q$-dimensional linear subspaces of $\BBW$ containing both the line $x$ and the line $x'$.
This set is isomorphic to $ F_{n-1}(q-2)$ if $x\neq x'$ and to $F_{n}(q-1)$ if $x=x'$.
\end{proof}

Of course, this formula is  interesting only when $\mu_n(q-1) \neq \mu_{n-1}(q-2))$.

\subsection{Application: the Radon transform}\label{subsection:Radon}
This section is extracted from~\cite{Sc95}. Recall that $n\geq2$.

One can roughly describe the Radon transform as follows. How to reconstruct a function (say with compact support)  on a real vector space $\BBV$ from the knowledge of its integral along all affine hyperplanes? Since the family of these hyperplanes (including the hyperplane at infinity) is given by the dual projective space $\BBP^*$, where $\BBP$ is the projective compactification  of $\BBV$, it is natural  to replace  $\BBV$ with $\BBP$.

We have $F_{n+1}(1)=\BBP_n$, the $n$-dimensional projective space and $F_{n+1}(n)=\BBP^*_n$, the dual projective space.
The Radon transform thus corresponds to the case $p=1$, $q=n$.  

With the preceding notations, the incidence relation $S$ is given by
\eqn
&&S= F_{n+1}(1,n)=\{(x,y)\in\BBP_n\times\BBP^*_n; \langle x,y\rangle=0\}.
\eneqn
The Radon transform of $\phi\in\CF(\BBP_n)$, an element of $\CF(\BBP^*_n)$, is defined by
\eq\label{eq:rsphi}
&&\shr_{(n+1;1,n)}(\phi)=\int_{\BBP_n}\un_S\cdot q_1^*\phi=\phi\conv \un_S.
\eneq
For $y\in\BBP^*_n$, we shall denote by $h_y$ its image in $\BBP_n$ by the incidence relation:
\eqn
&&h_y=\{x\in\BBP_n, \langle x,y\rangle=0\}.
\eneqn
Therefore, 
\eqn
&&\shr_{(n+1;1,n)}(\phi)(y)=\int_{\BBP_n} \phi\cdot\un_{h_y}.
\eneqn
Recall that the Euler-Poincar\'e index of $\BBP_n$ is
given by the formula:

\eq\label{eq:chiPn}
\chi (\BBP_n) = \left \{
\begin{array}{ll}
1 & \mbox{ if $n$ is even,}\\
0 & \mbox{ if $n$ is odd.}
\end{array}
\right. 
\eneq
Applying Proposition~\ref{pro:invforA}  together with~\eqref{eq:chiPn}, we get:

\begin{corollary}
Let $\phi \in \CF(\BBP_n)$. Then:
\eqn
\shr_{(n+1,n,1)} \circ \shr_{(n+1;1,n)} (\phi) = \left \{
\begin{array}{ll}
\phi & \mbox{ if $n$ is odd,}\\
-\phi +\int_{\BBP_n}\phi & \mbox{ if $n$ is even.}
\end{array}
\right. 
\eneqn
\end{corollary}

Now assume $\dim \BBV = 3$ and let us calculate the Radon transform of the
characteristic function ${\bf 1}_K$ of a compact subanalytic subset $K$ of $\BBV$ (see~\eqref{eq:chiK}). First, consider a
compact subanalytic subset $L$ of a two dimensional affine vector space $W$. By Poincar\'e's duality,
there is an isomorphism $\mbox{H}^1_L(W;\Q_W) \simeq \mbox{H}^1(L;\Q_L)$ and moreover there is a
short exact  sequence:
$$0\to \mbox{H}^0(W;\Q_W)\to \mbox{H}^0(W\setminus L;\Q_W)\to \mbox{H}^1_L(W;\Q_W) \to 0,$$
from which one deduces that:
\eqn
&& \rb_1(L) = \rb_0(W\setminus L) - 1,
\eneqn
where $\rb_i$ is the $i$-th Betti number.
Note that $\mbox{b}_0(W\setminus L)$ is the number of connected components of $W\setminus L$, hence 
$\mbox{b}_1(L)$ is the ``number of holes'' of the compact set $L$. We may sumarize:
\begin{corollary}
The value at $y\in\BBP^*_3$ of the Radon transform of ${\bf 1}_K$ is the number of connected components of
$K\cap h_y$ minus the number of its holes. 
\end{corollary}

The inversion formula of the Radon transform tells us how to reconstruct the set $K$ from the
knowledge of the number of connected components and holes of all its affine slices.

\section*{}
{\bf Conflict of interest}\\
The author declares that he has no conflict of interest.

\providecommand{\bysame}{\stLeavevmode\hbox to3em{\hrulefill}\thinspace}

\vspace*{1cm}
\noindent
\parbox[t]{21em}
{\scriptsize{
Pierre Schapira\\
Sorbonne Universit{\'e} and Universit{\'e} Paris Cit{\'e}, CNRS, IMJ-PRG\\
Campus Pierre et Marie Curie\\
F-75005 Paris France\\
e-mail: pierre.schapira@imj-prg.fr\\
http://webusers.imj-prg.fr/\textasciitilde pierre.schapira/
}}
\end{document}